\documentclass{article}
\usepackage{amsthm}
\usepackage{algorithm, algorithmicx, algpseudocode}
\usepackage{amsthm}
\usepackage{amsmath}
\usepackage{PRIMEarxiv}
\usepackage[utf8]{inputenc} 
\usepackage[T1]{fontenc}    
\usepackage{hyperref}       
\usepackage{url}            
\usepackage{booktabs}       
\usepackage{amsfonts}       
\usepackage{nicefrac}       
\usepackage{microtype}      
\usepackage{lipsum}
\usepackage{fancyhdr}       
\usepackage{graphicx}       
\graphicspath{{media/}} 
\newtheorem{lemma}{Lemma}
\newtheorem{theorem}{Theorem}
\title{Superlinear Convergence of GMRES for clustered eigenvalues and its application to least squares problems
}

\author{
   Zeyu LIAO \\
  Joint MSU-BIT-SMBU Research Center of Applied Mathematics \\
  Shenzhen MSU-BIT University \\
  Shenzhen\\
  \texttt{zeyuliao@smbu.edu.cn} \\
   \And
  Ken HAYAMI \\
  Professor Emeritus, National Institute of Informatics \\
  and The Graduate University for Advanced Studies 
(SOKENDAI) \\
  Tokyo\\
  \texttt{hayami@nii.ac.jp} \\
}

\begin{document}
\maketitle

\begin{abstract}
The objective of this paper is to understand the superlinear convergence behavior of the GMRES method when the coefficient matrix has clustered eigenvalues. In order to understand the phenomenon, we analyze the convergence using the Vandermonde matrix which is defined using the eigenvalues of the coefficient matrix. Although eigenvalues alone cannot explain the convergence, they may provide an upper bound of the residual, together with the right hand side vector and the eigenvectors of the coefficient matrix. We show that when the coefficient matrix is diagonalizable, 
if the eigenvalues of the coefficient matrix are clustered, the upper bound of the convergence curve shows superlinear convergence, when the norm of the matrix obtained by decomposing the right hand side vector into the eigenvector components is not so large. We apply the analysis to explain the convergence of inner-iteration preconditioned GMRES for least squares problems.
\end{abstract}


\section{Introduction}\label{sec1}
The generalized minimal residual method (GMRES) \cite{saad1986}, is a robust iterative method for the numerical solution of nonsymmetric square systems of linear equations. GMRES is a generalization of the MINRES method \cite{paige1975} which was developed for symmetric systems.\par
GMRES generates a Krylov subspace, and finds the solution in the Krylov subspace by minimizing the 2-norm of the residual. Consider the problem with square coefficient matrix
\begin{equation*}
    A\boldsymbol{x}=\boldsymbol{b},\qquad A\in \mathbb{R}^{n\times n},\qquad \boldsymbol{b}\in \mathbb{R}^{n}.
\end{equation*}
 Let $\boldsymbol{x_0}$ be the initial iterate. Denote the residual by $\boldsymbol{r}=\boldsymbol{b}-A\boldsymbol{x}$. Then, the initial residual is given by 
$\boldsymbol{{r}_0}=\boldsymbol{b}-A\boldsymbol{x_0}$. Define the Krylov subspace by
\begin{equation*}
    \mathcal{K}_k(A, \boldsymbol{{r}_0})=\textrm{span}\it\{ \boldsymbol{r_0}, A \boldsymbol{r_0}, \dots, A^{k-\rm 1\it} \boldsymbol{r_0}\}.
\end{equation*}
At each step of GMRES, we seek $\boldsymbol{z_k}\in  \mathcal{K}_k(A, \boldsymbol{{r}_0})$ and $\boldsymbol{x}_k=\boldsymbol{x_0}+\boldsymbol{z_k}$, such that the 2-norm of the residual $\boldsymbol{r_k}=\boldsymbol{b}-A\boldsymbol{x}_k=\boldsymbol{b}-A(\boldsymbol{x_0}+\boldsymbol{z_k})=\boldsymbol{{r}_0}-A\boldsymbol{z_k}$ is minimized, i.e. $\min_{\boldsymbol{z_k}\in \mathcal{K}_k(A, \boldsymbol{{r}_0})}\|\boldsymbol{{r}_0}-A\boldsymbol{z_k}\|_2$.\par
GMRES minimizes the residual on an expanding Krylov subspace. Thus, the residual decreases monotonically.\par

Let $q_m(x)$ be the minimal polynomial of the nonsingular matrix $A$, i.e., 
\begin{equation*}
    \mathcal{O}=q_m(A)=\alpha_{0} \rm{I}+\alpha_{1} \it A+\dots+\alpha_{m}A^{m},
\end{equation*}
where $\rm I$ is the identity matrix and $\alpha_0\neq 0$. It follows that
\begin{equation*}
    A^{-1}=-\frac{1}{\alpha_0}\sum_{j=0}^{m-1}\alpha_{j+1}A^j
\end{equation*}
This representation of $A^{-1}$ characterizes $\boldsymbol{x}=A^{-1}\boldsymbol{b}$ as a member of a Krylov subspace \cite{Ipsen1998The}.\par
One reason why the Krylov subspace method is efficient is that the solution of the linear system $A\boldsymbol{x}=\boldsymbol{b}$ may belong to a Krylov subspace of degree much less than the size of $A$.\par
The degree of the polynomial with respect to $A$ and $\boldsymbol{b}$, can be even smaller than the degree of the minimal polynomial of $A$.\par
Note that
\begin{equation*}
    \max_{\|\boldsymbol{{r}_0}\|=1}\min_{p\in \pi_k}\|p(A)\boldsymbol{{r}_0}\|\leq  \max_{\|\boldsymbol{{r}_0}\|=1}\min_{p\in \pi_k}\|p(A)\|\|\boldsymbol{{r}_0}\|=\min_{p\in \pi_k}\|p(A)\|\leq\min_{1\leq i_1<\cdots<i_k\leq n}\|(A-\lambda_{i_1}\rm{I}\it)\cdots(A-\lambda_{i_k}\rm{I})\|,
\end{equation*}
where $\pi_k$ denotes the set of polynomials of degree at most $k$ and with value $1$ at the origin, and $\boldsymbol{{r}_0}=\boldsymbol{b}-A\boldsymbol{x_0}$ \cite{tichy2004worst}.
This inequality shows that the norm of a polynomial of $A$ bounds the residual if the $\|\boldsymbol{{r}_0}\|$ is a constant. We can select the polynomial factors of the characteristic polynomial, which is determined by the eigenvalues. This implies that the eigenvalues can determine an upper bound of the residual.
\par
On the other hand, there is a famous result that any non-increasing convergence curve (residual norm versus iterations) is possible with the matrix $A$ which can be chosen to have any desired eigenvalues \cite{Anne1996Any}. See also \cite{1998Krylov}. This implies that the eigenvalues alone cannot determine the convergence of the residual.\par
Is there a conflict between the above two results? Can the eigenvalues decide the convergence behavior of GMRES\thinspace? In fact, there is no conflict between them. The eigenvalues give an upper bound for the non-increasing residual curve together with the right hand side vector and the eigenvectors. The convergence is influenced by the distribution of the eigenvalues, the projection of the right hand sides $\boldsymbol{b}$ on each eigenvectors and the normality of the matrix $A$.

The maximum number of steps to converge is governed by the degree of the minimal polynomial of $A$, which determines the inverse of $A$. Moreover, with a certain $\boldsymbol{b}$, the degree is less. Within the degree, it can even converge faster to a lower level and stagnate after the degree, when the eigenvalues of the coefficient matrix cluster, such as in the case of inner-iteration preconditioned GMRES for least squares problems \cite{Morikuni15}.\par

Although any non-increasing curve can occur, convergence can be improved by clustering the distribution of eigenvalues of $A$, when the right hand side vector is fixed. See Ipsen's work \cite{ipsen2000expressions} for how the projection of $\boldsymbol{b}$ onto the eigen-space affects the residual bounds.  The Ritz values of GMRES are also related to the convergence of GMRES \cite{1993The}.

The normality of $A$ tends to help the convergence. Large condition number $\kappa(V)$ of $V$, the matrix consisting of the eigenvectors of $A$, can hinder the convergence \cite{saad1986}.\par

Our work mainly shows that the clustering of eigenvalues has an effect of lowering the degree of the minimal polynomial of $A$ with respected to $\boldsymbol{{r}_0}$, make the GMRES converge faster, when $A$ is diagonalizable. This is demonstrated in the example of the inner-iteration preconditioned GMRES for least squares problems.\par
The organization of the paper is as follows.  In section 2, we present a brief analysis of the upper bound of the residual. In section 3, we apply the analysis of section 2 to the case of clustered eigenvalues. In section 4, we offer a concise review of the inner-iteration preconditioned GMRES method for least squares problems, and provide a thorough analysis of the method by the framework established. In section 5, we present an illustrative example constructed based on the work of Greenbaum et al.  \cite{Anne1996Any}. This example is then interpreted using the theorem established in the preceding section. Finally, in section 6, we conclude the paper.

\section{Upper bound of the residual}
\begin{theorem}\label{theroyfirst}
   Let  $A\in\mathbb{R}^{n\times n}$ be diagonalizable ($A$ could be singular), $b\in \mathcal{R}(A)$, where $\mathcal{R}(A)$ is the range space of $A$,  and $d$ is the grade of $\mathcal{K}_{k}$($A, \boldsymbol{r_0}$). Let $\boldsymbol{r_0}=c_1\boldsymbol{v_1}+c_2\boldsymbol{v_2}+\dots+c_d\boldsymbol{v_d},$ where $v_i\in \mathbb{C}^{n}\; (\|v_i\|_2=1)$ are the eigenvectors of $A$ corresponding to nonzero eigenvalues $\lambda_i\in\mathbb{C}$ $(i=1, 2, \dots, d)$ and $c_i$ are the corresponding weights. Assume that $\lambda_i\neq \lambda_j$ if $i\neq j$. \\
   Then,
   \begin{align}\label{eqVandermonde}
    \min_{\boldsymbol{x}_k=\boldsymbol{x_0}+\boldsymbol{z_k},\; \boldsymbol{z_k}\in\mathcal{K}_k(A, 
 \boldsymbol{{r}_0})}\|\boldsymbol{b}-A\boldsymbol{x}_k\|_2
      &\leq\|V_d\textsc{diag}[c_1, c_2, \dots, c_d]\|_2\min_{\boldsymbol{y_k}\in\mathbb{C}^k}\|\Lambda_d^k \boldsymbol{y_k}-[1, 1, \dots, 1]^{\top}\|_2.
\end{align}

    where $V_d=[\boldsymbol{v_1}, \boldsymbol{v_2}, \dots, \boldsymbol{v_d}]$, $\boldsymbol{x_k}=\boldsymbol{x_0}+\boldsymbol{z_k}=\boldsymbol{x_0}+[\boldsymbol{r_0}, A\boldsymbol{r_0}, \dots, A^{k-1}\boldsymbol{r_0}]\boldsymbol{y_k}$ and
\begin{equation}\label{eqLambda}
    \Lambda_d^k=\left(
  \begin{array}{cccc}
   \lambda_1  &\lambda_1^2 & \cdots &\lambda_1^k  \\
   \lambda_2  & \lambda_2^2 & \cdots &\lambda_2^k  \\
     \ldots  &\ldots &\ldots &\ldots \\\
    \lambda_d & \lambda_d^2 & \cdots  & \lambda_d^k \\
  \end{array}
\right)\in \mathbb{C}^{d\times k}
\end{equation}
is a Vandermonde matrix.
\end{theorem}
\begin{proof}
     Consider the system of linear equations
\begin{equation}\label{eqfirst1}
    A\boldsymbol{x}=\boldsymbol{b},\qquad A\in \mathbb{R}^{n\times n},\qquad \boldsymbol{b}\in \mathbb{R}^{n}, \qquad \boldsymbol{b}\in \mathcal{R}(A).
\end{equation}\par
Assume  $A\in\mathbb{R}^{n\times n}$ is diagonalizable ($A$ could be singular), $\boldsymbol{r_0}=\boldsymbol{b}-A\boldsymbol{x_0}$, and $d$ is the grade of $\mathcal{K}_{k}$($A, \boldsymbol{r_0}$), which means $d$ is the smallest integer such that
$K_d(A,\boldsymbol{r_0})=K_{d+1}(A, \boldsymbol{r_0})$.
Let $\boldsymbol{r_0}=c_1\boldsymbol{v_1}+c_2\boldsymbol{v_2}+\dots+c_d\boldsymbol{v_d},$ where $\boldsymbol{v_i}\in \mathbb{C}^{n}\; (\|\boldsymbol{v_i}\|_2=1)$ is the eigenvector of $A$ corresponding to nonzero eigenvalue of $\lambda_i\in\mathbb{C}$ $(i=1, 2, \dots, d)$ and $c_i$ is the corresponding weight. Assume that $\lambda_i (i=1,2,\dots,d)$ are distinct. Thus, $\boldsymbol{r_0}$ has the following representation.
\begin{equation}\label{eqsecond2}
    \boldsymbol{r_0}=[c_1\boldsymbol{v_1},c_2\boldsymbol{v_2},\dots,c_d\boldsymbol{v_d}][1,\cdots,1]^{\top}.
\end{equation}
Then
\begin{align*} A\boldsymbol{r_0}&=A(c_1\boldsymbol{v_1}+c_2\boldsymbol{v_2}+\dots+c_d\boldsymbol{v_d})=c_1A\boldsymbol{v_1}+c_2A\boldsymbol{v_2}+\dots+c_dA\boldsymbol{v_d}
\\
&=c_1\lambda_1\boldsymbol{v_1}+c_2\lambda_2\boldsymbol{v_2}+\dots+c_d\lambda_d\boldsymbol{v_d}\\
&=[c_1\boldsymbol{v_1},c_2\boldsymbol{v_2},\dots,c_d\boldsymbol{v_d}][\lambda_1, \lambda_2, \dots,\lambda_d]^{\top}.
\end{align*}
Similarly, 
\begin{equation*}
        A^k\boldsymbol{r_0}=[c_1\boldsymbol{v_1},c_2\boldsymbol{v_2},\dots,c_d\boldsymbol{v_d}][\lambda_1^k, \lambda_2^k, \dots,\lambda^k_d]^{\top}.
\end{equation*}
    Hence, for the $k$th iterate $\boldsymbol{x_k}=\boldsymbol{x_0}+[\boldsymbol{r_0}, A\boldsymbol{r_0},\dots, A^{k-1}\boldsymbol{r_0}]\boldsymbol{y_k}$ of $A\boldsymbol{x}=\boldsymbol{b}$,\\
\begin{align*}
        \min_{\boldsymbol{x}_k=\boldsymbol{x_0}+\boldsymbol{z_k},\; \boldsymbol{z_k}\in\mathcal{K}_k(A, \boldsymbol{{r}_0})} \|\boldsymbol{b}-A\boldsymbol{x}_k\|_2&=\min_{\boldsymbol{y_k}\in\mathbb{C}^k} \|\boldsymbol{r_0}-A[\boldsymbol{r_0}, A\boldsymbol{r_0},\dots, A^{k-1}\boldsymbol{r_0}]\boldsymbol{y_k}\|_2\\
         &=\min_{\boldsymbol{y_k}\in\mathbb{C}^k} \|[c_1\boldsymbol{v_1},c_2\boldsymbol{v_2},\dots,c_d\boldsymbol{v_d}][1, 1, \dots, 1]^{\top}-[c_1\boldsymbol{v_1}, c_2\boldsymbol{v_2}, \dots,c_d\boldsymbol{v_d}]\Lambda_d^k \boldsymbol{y_k}\|_2\\
         &=\min_{\boldsymbol{y_k}\in\mathbb{C}^k} \|[c_1\boldsymbol{v_1}, c_2\boldsymbol{v_2}, \dots,c_d\boldsymbol{v_d}]([1, 1, \dots, 1]^{\top}-\Lambda_d^k \boldsymbol{y_k})\|_2
\end{align*}
    where
\begin{equation}\label{eqLambda}
    \Lambda_d^k=\left(
  \begin{array}{cccc}
   \lambda_1  &\lambda_1^2 & \cdots &\lambda_1^k  \\
   \lambda_2  & \lambda_2^2 & \cdots &\lambda_2^k  \\
     \ldots  &\ldots &\ldots &\ldots \\\
    \lambda_d & \lambda_d^2 & \cdots  & \lambda_d^k \\
  \end{array}
\right)\in \mathbb{C}^{d\times k}
\end{equation}
is a Vandermonde matrix.\par
Thus, we have
\begin{align}
    \min_{\boldsymbol{x}_k=\boldsymbol{x_0}+\boldsymbol{z_k},\; \boldsymbol{z_k}\in\mathcal{K}_k(A,\boldsymbol{{r}_0})}\|\boldsymbol{b}-A\boldsymbol{x}_k\|_2&=\min_{\boldsymbol{y_k}\in\mathbb{C}^k}\|[c_1\boldsymbol{v_1},c_2\boldsymbol{v_2},\dots,c_d\boldsymbol{v_d}](\Lambda_d^k \boldsymbol{y_k}-[1, 1, \dots, 1]^{\top})\|_2\\
      &\leq\min_{\boldsymbol{y_k}\in\mathbb{C}^k}\|[c_1\boldsymbol{v_1},c_2\boldsymbol{v_2},\dots,c_d\boldsymbol{v_d}]\|_2\|\Lambda_d^k \boldsymbol{y_k}-[1, 1, \dots, 1]^{\top}\|_2 \nonumber\\
      &=\|V_d\text{diag}[c_1, c_2, \dots, c_d]\|_2\min_{\boldsymbol{y_k}\in\mathbb{C}^k}\|\Lambda_d^k \boldsymbol{y_k}-[1, 1, \dots, 1]^{\top}\|_2 .\label{eqVandermonde}
\end{align}
where $V_d=[\boldsymbol{v_1}, \boldsymbol{v_2}, \dots, \boldsymbol{v_d}]$.
\end{proof}
For the case when $\lambda_i (i=1, 2, \dots, d)$ are not distinct, we can analyze similarly.
For example, 
for $\boldsymbol{r_0}=c_1\boldsymbol{v_1}+c_2\boldsymbol{v_2}+\dots+c_5\boldsymbol{v_5}$, where $\boldsymbol{v_1}, \boldsymbol{v_2}, \boldsymbol{v_3}$ correspond to $\lambda_1$, $\boldsymbol{v_4}, \boldsymbol{v_5}$ correspond to $ \lambda_2$ Then,
\begin{align*}
A\boldsymbol{r_0}&=A(c_1\boldsymbol{v_1}+c_2\boldsymbol{v_2}+\dots+c_5\boldsymbol{v_5})\\
&=c_1\lambda_1\boldsymbol{v_1}+c_2\lambda_1\boldsymbol{v_2}+c_1\lambda_1\boldsymbol{v_3}+c_4\lambda_2\boldsymbol{v_4}+c_5\lambda_2\boldsymbol{v_5}\\
&=(c_1\boldsymbol{v_1}+c_2\boldsymbol{v_2}+c_3\boldsymbol{v_3})\lambda_1+(c_4\boldsymbol{v_4}+c_5\boldsymbol{v_5})\lambda_2\\
&=[\widetilde{c_1}\boldsymbol{\widetilde{v_1}}, \widetilde{c_2}\boldsymbol{\widetilde{v_2}}][\lambda_1, \lambda_2]^\top.
\end{align*}
This example shows that the multiple eigenvalues merge, and only the distinct eigenvalues contribute in constructing the Vandermonde coefficient matrix.

Note that, for diagonalizable matrix $A$, Saad \cite{saad1986} gives the bound
\begin{equation}\label{saadbound}
  \frac{\|\boldsymbol{r_k}\|_2}{\|\boldsymbol{{r}_0}\|_2}\leq \kappa(V)\max_{i=1,2,\dots, n}|p(\lambda_i)|,
\end{equation}
where $p$ is any polynomial of degree $\leq k$ which satisfies the constraint $p(0)=1$, and the $k$th iterate vector $\boldsymbol{x_k}\in \boldsymbol{x_0}+\mathcal{K}_d(A, \boldsymbol{{r}_0})$ is associated with the residual vector $\boldsymbol{r_k}=\boldsymbol{b}-A
\boldsymbol{x_k}=p(A)\boldsymbol{{r}_0}$,
and the eigenvalue decomposition of $A$ is given by $A=VDV^{-1}$ where $D$ is a diagonal matrix consisting of the eigenvalues of $A$, $V$ is a matrix whose columns are eigenvectors of $A$. The bound contains $\kappa(V)=\|V\|_2\|V^{-1}\|_2$. However, our bound (\ref{eqVandermonde}) does not contain $\kappa(V)$. It contains 
$\|[c_1\boldsymbol{v_1},c_2\boldsymbol{v_2},\dots,c_d\boldsymbol{v_d}]\|_2$,
and $\min_{\boldsymbol{y_k}\in\mathbb{C}^k}$$\|\Lambda_d^k \boldsymbol{y_k}-[1,\cdots,1]^{\top}\|_2$, which 
depends on $\Lambda_d^k$ instead of $V$. Moreover, our bound involves $c_i$, which is different from (\ref{saadbound}). \par
By passing, the following lemma holds for the case when $A$ is normal.
\begin{lemma}
     If $\boldsymbol{v_1}, \boldsymbol{v_2}, \dots, \boldsymbol{v_d}$ are orthonormal, then
     \begin{equation*}
   \min_{\boldsymbol{x}_k=\boldsymbol{x_0}+\boldsymbol{z_k},\; \boldsymbol{z_k}\in\mathcal{K}_k(A,\boldsymbol{{r}_0})} \frac{\|\boldsymbol{r_k}\|_2}{\|\boldsymbol{r_0}\|_2}\leq\frac{\|c\|_{\infty}}{\|c\|_2}\min_{\boldsymbol{y_k}\in\mathbb{C}^k}\|\Lambda_d^k \boldsymbol{y_k}-[1, 1, \dots, 1]^{\top}\|_2.
\end{equation*}
where $\boldsymbol{r_k}=\boldsymbol{b}-A\boldsymbol{x_k}.$
 \end{lemma}
\begin{proof}
    If $\boldsymbol{v_1}, \boldsymbol{v_2}, \dots, \boldsymbol{v_d}$ are orthonormal, then $\|\boldsymbol{r_0}\|_2=\sqrt{c_1^2+\dots+c_d^2}=\|c\|_2$, where  $c=[c_1, c_2, \dots, c_d]^{\top}$. Let $\boldsymbol{x}=[x_1, x_2, \dots, x_d]^{\top}.$\\
Then,
\begin{align*}
    \|[c_1\boldsymbol{v_1},c_2\boldsymbol{v_2},\dots,c_d\boldsymbol{v_d}]\|_2=&\max_{\|\boldsymbol{x}\|_2=1}\|[c_1\boldsymbol{v_1},c_2\boldsymbol{v_2},\dots,c_d\boldsymbol{v_d}][x_1, x_2, \dots, x_d]^{\top}\|_2 \\
    =&\max_{\|\boldsymbol{x}\|_2=1}\|c_1x_1\boldsymbol{v_1}+c_2x_2\boldsymbol{v_2}+\dots+c_dx_d\boldsymbol{v_d}\|_2 \\
   =&  \max_{\|\boldsymbol{x}\|_2=1}\{(c_1x_1)^2+(c_2x_2)^2+\cdots+(c_dx_d)^2\}^{\frac{1}{2}}  \\
    =&\max_{1\leq i\leq d}|c_i|\\
   =& \|c\|_{\infty},
\end{align*}
and we have
\begin{equation*}
   \min_{\boldsymbol{x}_k=\boldsymbol{x_0}+\boldsymbol{z_k},\; \boldsymbol{z_k}\in\mathcal{K}_k(A,\boldsymbol{{r}_0})} \frac{\|\boldsymbol{b}-A\boldsymbol{x_k}\|_2}{\|\boldsymbol{r_0}\|_2}= \min_{\boldsymbol{x}_k=\boldsymbol{x_0}+\boldsymbol{z_k},\; \boldsymbol{z_k}\in\mathcal{K}_k(A,\boldsymbol{{r}_0})} \frac{\|\boldsymbol{r_k}\|_2}{\|\boldsymbol{r_0}\|_2}\leq\frac{\|c\|_{\infty}}{\|c\|_2}\min_{\boldsymbol{y_k}\in\mathbb{C}^k}\|\Lambda_d^k \boldsymbol{y_k}-[1, 1, \dots, 1]^{\top}\|_2.
\end{equation*}
\end{proof}
\section{Clustered case}
Now we return to the general case where $A$ is not necessarily normal. Consider the case when the eigenvalues of $A$ have a clustered structure, where there are $s$ clusters, and each eigenvalue $\lambda_i\in \mathbb{C}, 1\leq i\leq d$ belongs to a cluster around a center $\gamma_j\neq 0\in \mathbb{C}$ with a small radius $\epsilon>0$, $i.e.$ $\lambda_i=\gamma_j+\epsilon_i$,\; $1\leq j\leq s$,\; $0\leq |\epsilon_i|\leq\epsilon\ll 1,$ where $\epsilon\equiv\max(|\epsilon_1|, |\epsilon_2|, \dots, |\epsilon_d|)$, and $\gamma_1, \gamma_2, \dots, \gamma_s$ are distinct, i.e. $\gamma_{j_1}\neq\gamma_{j_2}$ if ${j_1}\neq {j_2}$. Note that, the original cluster structure of eigenvalues will be adjusted according to the number of iterations $k$ which limits the number of cluster centers $\gamma_j$. Then, the following theorem holds.\par
\begin{theorem}\label{theoremtwo}
     \begin{align}
    \min_{\boldsymbol{x}_k=\boldsymbol{x_0}+\boldsymbol{z_k},\; \boldsymbol{z_k}\in\mathcal{K}_k(A,\boldsymbol{{r}_0})}\|\boldsymbol{b}-A\boldsymbol{x}_k\|_2
      &\leq\|V_d\textsc{diag}[c_1, c_2, \dots, c_d]\|_2\|(f(\lambda_1), f(\lambda_2), \dots, f(\lambda_d))^\top\|_2,
\end{align}
where  \begin{equation*}
       f(\gamma)=(-1)^{k-1}\frac{1}{\prod_{j=1}^k \gamma_j}\prod_{j=1}^k(\gamma-\gamma_j).
   \end{equation*}.
\end{theorem}
\begin{proof}
    Define the polynomial $f(\gamma)$ using distinct centers $\gamma_1, \gamma_2, \dots, \gamma_k$  as
     \begin{equation*}
       f(\gamma)=(-1)^{k-1}\frac{1}{\prod_{j=1}^k \gamma_j}\prod_{j=1}^k(\gamma-\gamma_j).
   \end{equation*}
    
    Expand the polynomial $f(\gamma)$ as
    \begin{equation*}
        f(\gamma)=y_1^k\gamma^k+y_1^{k-1}\gamma^{k-1}+\cdots+y_1^1\gamma-1.
    \end{equation*}
  Thus, $\;$$f(\gamma_j)=0$ is the $j$th equation of $\widetilde{\Lambda}_k y-[1,\cdots,1]^{\top}=0,$ and $y_1$ is the solution to this system, it follows that $\gamma_j$ $(j= 1, 2, \dots, k)$ are the roots of $f(\gamma)=0$, where\begin{equation*}
    \widetilde{\Lambda}_k=\left(
  \begin{array}{cccc}
   \gamma_1 &\gamma_1^2 & \cdots &\gamma_1^k\\
   \gamma_2 & \gamma_2^2 & \cdots &\gamma_2^k\\
     \ldots  &\ldots &\ldots &\ldots\\
    \gamma_k & \gamma_k^2 & \cdots  & \gamma_k^k\\
  \end{array}
\right) \in \mathbb{C}^{k\times k}.
\end{equation*}\\ From the expansion of the polynomial $f(\gamma)$, we get $\boldsymbol{y_1}=(y_1^1, y_1^2, \dots, y_1^k)$, where $y_1=\arg \min_{y\in \mathbb{C}^k} \|\widetilde{\Lambda}_k y-[1,\cdots,1]^{\top}\|_2$.\\ Then, we have  
   \begin{equation*}
       f(\lambda_i)=y_1^k\lambda_i^k+y_1^{k-1}\lambda_i^{k-1}+\cdots+y_1^1\lambda_i-1.
   \end{equation*}
   \begin{equation*}
       \min_{\boldsymbol{y_k}\in\mathbb{C}^k}\|\Lambda_d^k \boldsymbol{y_k}-[1, 1, \dots, 1]^{\top}\|_2 \leq\|\Lambda_d^k \boldsymbol{y_1}-[1, 1, \dots, 1]^{\top}\|_2=\|(f(\lambda_1), f(\lambda_2), \dots, f(\lambda_d))^\top\|_2.
   \end{equation*}
   Thus, from Theorem\ref{theroyfirst} we have \begin{align}
    \min_{\boldsymbol{x}_k=\boldsymbol{x_0}+\boldsymbol{z_k},\; \boldsymbol{z_k}\in\mathcal{K}_k(A,\boldsymbol{{r}_0})}\|\boldsymbol{b}-A\boldsymbol{x}_k\|_2
      &\leq\|V_d\textsc{diag}[c_1, c_2, \dots, c_d]\|_2\|(f(\lambda_1), f(\lambda_2), \dots, f(\lambda_d))^\top\|_2,
\end{align}\end{proof}
\begin{theorem}\label{theoremthree}
    Let $\lambda_i=\gamma_j+\epsilon_i, 1\leq j\leq s, 0\leq |\epsilon_i|\leq\epsilon\ll 1,$ $f(\gamma_j)=0$, where $\epsilon\equiv\max(|\epsilon_1|, |\epsilon_2|, \dots, |\epsilon_d|)$, and  \begin{equation*}
        f(\gamma)\equiv y_1^k\gamma^k+y_1^{k-1}\gamma^{k-1}+\cdots+y_1^1\gamma-1.
    \end{equation*}
    Then, we have
    \begin{equation*}
        \|(f(\lambda_1), f(\lambda_2), \dots, f(\lambda_d))^\top\|_2 = \mathcal{O}(\epsilon).
    \end{equation*}

\end{theorem}
\begin{proof}
Since $f(\lambda)$ is a $k$-th order polynomial, by Talyor expansion
    \begin{equation*}
    f(\lambda_i)=f(\gamma_j+\epsilon_i)=f(\gamma_j)+\sum_{l=1}^k\frac{f^{(l)}(\gamma_j)}{l!}\epsilon_i^l=\sum_{l=1}^k\frac{f^{(l)}(\gamma_j)}{l!}\epsilon_i^l=\mathcal{O}(\epsilon_i).
    \end{equation*}
        Since $0\leq \epsilon_i\leq\epsilon\ll 1$, we have
    \begin{equation*}
        \|(f(\lambda_1), f(\lambda_2), \dots, f(\lambda_d))^\top\|_2=\|(\mathcal{O}(\epsilon_1), \mathcal{O}(\epsilon_2),\dots, \mathcal{O}(\epsilon_d))^\top\|_2=\mathcal{O}(\epsilon).
    \end{equation*}
\end{proof}
At step $k$, replace $\lambda_i$ in (\ref{eqLambda}) by $\gamma_j+\epsilon_i$ to obtain
\begin{equation}\label{euqationdef61}\Lambda_\epsilon\equiv\Lambda_d^k=\left(
  \begin{array}{cccc}
   \gamma_1+\epsilon_1  &(\gamma_1+\epsilon_1)^2 & \cdots &(\gamma_1+\epsilon_1)^k\\
    \gamma_1+\epsilon_2  &(\gamma_1+\epsilon_2)^2 & \cdots &(\gamma_1+\epsilon_2)^k\\
    \cdots&\cdots&\cdots&\cdots\\
   \gamma_2+\epsilon_i  & (\gamma_2+\epsilon_i)^2 & \cdots &(\gamma_2+\epsilon_i)^k\\
     \cdots  &\cdots &\cdots &\cdots\\
    \gamma_s+\epsilon_d & (\gamma_s+\epsilon_d)^2 & \cdots& (\gamma_s+\epsilon_d)^k\\
  \end{array}
\right)\in \mathbb{C}^{d\times k}.
\end{equation}
Then,
\begin{equation*}
    \Lambda_\epsilon\approx \widetilde{\Lambda}_\epsilon
= \Lambda_s+P.
\end{equation*}

\begin{equation*}
    \widetilde{\Lambda}_\epsilon=\left(
  \begin{array}{cccc}
   \gamma_1+\epsilon_1  &\gamma_1^2+2\gamma_1\epsilon_1 & \cdots &\gamma_1^k+k\gamma_1^{k-1}\epsilon_1\\
   \gamma_1+\epsilon_2  &\gamma_1^2+2\gamma_1\epsilon_2 & \cdots &\gamma_1^k
   +k\gamma_1^{k-1}\epsilon_2\\
       \cdots&\cdots&\cdots&\cdots\\
   \gamma_2+\epsilon_i  & \gamma_2^2+2\gamma_2\epsilon_i & \cdots &\gamma_2^k+k\gamma_2^{k-1}\epsilon_i\\
     \cdots  &\cdots &\cdots &\cdots\\
    \gamma_s+\epsilon_d & \gamma_s^2+2\gamma_s\epsilon_d & \cdots  & \gamma_s^k+k\gamma_s^{k-1}\epsilon_d\\
  \end{array}
\right) \in\mathbb{C}^{d\times k}.
\end{equation*}

\begin{equation*}
    \Lambda_s=\left(
  \begin{array}{cccc}
   \gamma_1  &\gamma_1^2 & \cdots &\gamma_1^k\\
   \gamma_1 &\gamma_1^2 & \cdots &\gamma_1^k \\
       \cdots&\cdots&\cdots&\cdots\\
   \gamma_2 & \gamma_2^2 & \cdots &\gamma_2^k \\
     \cdots  &\cdots &\cdots &\cdots \\
    \gamma_s & \gamma_s^2 & \cdots  & \gamma_s^k\\
  \end{array}
\right)\in\mathbb{C}^{d\times k},\qquad P=\left(
  \begin{array}{cccc}
   \epsilon_1  &2\gamma_1\epsilon_1 & \cdots &k\gamma_1^{k-1}\epsilon_1 \\
   \epsilon_2  & 2\gamma_1\epsilon_2 & \cdots&k\gamma_1^{k-1}\epsilon_2 \\
     \cdots  &\cdots &\cdots &\cdots\\
    \epsilon_d & 2\gamma_s\epsilon_d & \cdots  & k\gamma_s^{k-1}\epsilon_d\\
  \end{array}
\right)\in\mathbb{C}^{d\times k}.
\end{equation*}
Deleting identical rows of $\Lambda_s$, we obtain $\widetilde{\Lambda}_s$,

Hence, denote the $i$th row of $P$ as 
$[\epsilon_i\quad 2\gamma_{j(i)}\epsilon_i\quad \dots \quad k\gamma^{k-1}_{j(i)}\epsilon_i]$ to indicate that the elements of the $i$th row are related to the center $\gamma_{j(i)}$.
\begin{lemma}\label{lemmaesecond}
      Let $A$ have a cluster of eigenvalues $\lambda_i, 1\leq i\leq d$ around $\gamma_{j(1)}=1$ with $\lambda_i=1+\epsilon_i$, where the eigenvalues $\lambda_i$ are numbered in descending order of $|\epsilon_i|$. Then, at the $k$th iteration $(k>1)$, we have
\begin{align} 
      \|r_k\|_2 \leq \|V_d\textsc{diag}[c_1, c_2, \dots, c_d]\|_2\;\|(f(\lambda_k), f(\lambda_{k+1}), \dots, f(\lambda_d))^\top\|_2,
\end{align}
where  \begin{equation*}
       f(\gamma)=(-1)^{k-1}\frac{\gamma-1}{\prod_{i=1}^{k-1} \lambda_i}\prod_{i=1}^{k-1}(\gamma-\lambda_i).
   \end{equation*}
  For $k=1$, $f(\gamma)=\gamma-1$ . Here, the  
     first order approximation of $\|(f(\lambda_k), f(\lambda_{k+1}), \dots, f(\lambda_d))^\top\|_2$ in terms of $\epsilon_k$  is given by $\epsilon_k\sqrt{n-k+1}\;\Pi_{i=1}^{k-1}\frac{1-\lambda_i}{\lambda_i}$.
\end{lemma}
\begin{proof}
Set \begin{equation*}
       f(\gamma)=(-1)^{k-1}\frac{\gamma-1}{\prod_{i=1}^{k-1} \lambda_i}\prod_{i=1}^{k-1}(\gamma-\lambda_i).
   \end{equation*}
At step $k$,
by Theorem\ref{theoremtwo} we obtain
\begin{align*}
    \|r_{k}\|_2&=  \|(b-Ax_{k})\|_2\\
  &\leq \|V_d\textsc{diag}[c_1, c_2, \dots, c_d]\|_2\|(f(\lambda_1), f(\lambda_2), \dots, f(\lambda_d))^\top\|_2.
\end{align*}
Since $f(\lambda_1), \dots,f(\lambda_{k-1})=0$, 
\begin{equation*}
    \|(f(\lambda_1), f(\lambda_2), \dots, f(\lambda_d)^\top\|_2=\|(f(\lambda_k), f(\lambda_{k+1}), \dots, f(\lambda_d))^\top\|_2.
\end{equation*}
\\
Then,
      \begin{align*}
   \|V_d\text{diag}[c_1, c_2, \dots, c_d]\|_2 \|\Lambda_\epsilon y-[1, 1, \dots, 1]^{\top}\|_2&\simeq \|V_d\text{diag}[c_1, c_2, \dots, c_d]\|_2 \|Py_1\|_2\\&<\|V_d\text{diag}[c_1, c_2, \dots, c_d]\|_2 \sqrt{n-k}\epsilon|f'(1)|\\
  &=\|V_d\text{diag}[c_1, c_2, \dots, c_d]\|_2\epsilon\sqrt{n-k}\;\Pi_{i=1}^k\frac{1-\lambda_i}{\lambda_i}.
\end{align*} 
\end{proof}   
Consider the case when the eigenvalues of $A$ have a clustered structure, where there are $s$ clusters, and each eigenvalue $\lambda_i\in \mathbb{C}, 1\leq i\leq d$ belongs to a cluster around a center $\gamma_j\in \mathbb{C}$ with a small radius $\epsilon$, $i.e.$ $\lambda_i=\gamma_j+\epsilon_i, 1\leq j\leq s, 0\leq |\epsilon_i|\leq\epsilon\ll 1,$ where $\epsilon\equiv\max(|\epsilon_1|, |\epsilon_2|, \dots,|\epsilon_d|)$, and $\gamma_1, \gamma_2, \dots, \gamma_s$ are distinct, i.e. $\gamma_{j_1}\neq\gamma_{j_2}$ if ${j_1}\neq {j_2}$. \par

\begin{lemma}
 For the case when there are many centers, assume $A$ has $s$ centers. Let the distance of an eigenvalue $\lambda_i$ from the cluster center $\gamma_{j(i)}$  be $\epsilon_i$ $(\epsilon_i=|\lambda_i-\gamma_{j(i)}|)$, where the eigenvalues are $\lambda_i, 1\leq i<d$ corresponding to the center $\gamma_j(i)$ with a descending order of $\epsilon_i$. Then, at the $k$th iteration $(k>s)$, 
\begin{align} 
      \|r_k\|_2\leq\|V_d\textsc{diag}[c_1, c_2, \dots, c_d]\|_2\|(f(\lambda_k), f(\lambda_{k+1}), \dots, f(\lambda_d))^\top\|_2,
\end{align}
where  \begin{equation*}
       f(\gamma)=(-1)^{k-1}\frac{1}{\prod_{i=1}^{k-s} \lambda_i\prod_{j=1}^s\gamma_j}\prod_{i=1}^{k-s}(\gamma-\lambda_i)\prod_{j=1}^s(\gamma-\gamma_j).
   \end{equation*}
  For $k=s$, $f(\gamma)=(-1)^s\frac{1}{\prod_{j=1}^{s}\gamma_j}\prod_{j=1}^s(\gamma-\gamma_j)$, and its 
     first order approximation in terms of $\epsilon_k$ is given by 
     \begin{equation}
         \epsilon_k\|V\textsc{diag}[c_1, c_2, \dots, c_d]\|_2\|(f'(\gamma_{j(1)}), f'(\gamma_{j(2)}), \dots,f'(\gamma_{j(d)}))^{\top}\|_2.\label{estimationcenter}
     \end{equation}
     or
      \begin{equation}\label{estimationcenter}
         \epsilon_k\|V\textsc{diag}[c_1, c_2, \dots, c_d]\|_2\|(f'(\gamma_{j(k)}), f'(\gamma_{j(k+1)}), \dots,f'(\gamma_{j(d)}))^{\top}\|_2,
     \end{equation}
\end{lemma}
Note that, the estimation (\ref{estimationcenter}) contains only centers $\gamma_{j(i)}$ for which all $i$ belong to $j$ with $\epsilon_i\neq 0$. 
\begin{proof}
    Here, the $\Lambda_\epsilon$, $\widetilde{\Lambda}_k$ and $P$  were defined before Lemma\ref{lemmaesecond}.
    \begin{align*}
     \min_{y\in\mathbb{C}^k}\|\Lambda_\epsilon y-[1, 1, \dots, 1]^{\top}\|_2&\leq \|\Lambda_\epsilon y_1-[1, 1, \dots, 1]^{\top}\|_2\\ &\approx\|\Lambda_s y_1-[1, 1, \dots, 1]^{\top}+Py_1\|_2 \\&\leq \|\Lambda_s y_1-[1, 1, \dots, 1]^{\top}\|_2+\|Py_1\|_2\\ &= \|\widetilde{\Lambda}_s y_1-[1, 1, \dots, 1]^{\top}\|_2+\|Py_1\|_2=\|Py_1\|_2 \quad (k=s),
 \end{align*} 
 since $\widetilde{\Lambda}_k$ is nonsingular, because $\gamma_1, \gamma_2, \dots, \gamma_k$ are distinct.\\
Let $y_1=(y_1^1, y_1^2, \cdots, y_1^k)^{\top}$. Then,
\begin{equation*}
    (Py_1)_{i}=(ky_1^{k}\gamma_{j(i)}^{k-1}+(k-1)y_1^{k-1}\gamma_{j(i)}^{k-2}+\cdots+y_1^1)\epsilon_i,\qquad 1\leq i\leq d.
\end{equation*}
     
    \begin{equation*}
        f'(\gamma_{j(i)})=(-1)^{k-1}\frac{1}{\prod_{j=1}^k \gamma_j}\prod_{j=1,\cdots,k,\; j\neq j(i)}(\gamma_{j(i)}-\gamma_j),
    \end{equation*}
    \begin{equation*}
        f'(\gamma_{j(i)})=ky_1^{k}\gamma_{j(i)}^{k-1}+(k-1)y_1^{k-1}\gamma_{j(i)}^{k-2}+\cdots+y_1^1,
    \end{equation*}
    and 
    \begin{equation*}
        (Py_1)_i=f'(\gamma_{j(i)})\epsilon_i.
    \end{equation*}
    Thus, at step $k$, we have
\begin{align}
  \min_{\boldsymbol{x}_k=\boldsymbol{x_0}+\boldsymbol{z_k},\; \boldsymbol{z_k}\in\mathcal{K}_k(A,\boldsymbol{{r}_0})}  \|\boldsymbol{b}-Ax_k\|_2&\leq\|V\text{diag}[c_1, c_2, \dots, c_d]\|_2\min_{y\in\mathbb{C}^k}\|\Lambda_d^k y-[1,\dots,1]^{\top}\|_2 \nonumber\\ &\approx \|V\text{diag}[c_1, c_2, \dots, c_d]\|_2\|(f'(\gamma_{j(1)})\epsilon_1, f'(\gamma_{j(2)})\epsilon_2, \dots,f'(\gamma_{j(d)})\epsilon_d)^{\top}\|_2\nonumber\\
  &\leq\epsilon\|V\text{diag}[c_1, c_2, \dots, c_d]\|_2\|(f'(\gamma_{j(1)}), f'(\gamma_{j(2)}), \dots,f'(\gamma_{j(d)}))^{\top}\|_2.
\end{align}
\end{proof}
\section{Inner-iteration preconditioning}
\subsection{Inner-iteration by NR-SOR for BA-GMRES}
Hayami et al. \cite{Hayami10} proposed preconditioning the $m\times n$ rectangular matrix $A$ of the least squares problem 
\begin{equation}\label{eqstar}
    \min_{x\in\mathbb{R}^{n}}\|b-Ax\|_2,\qquad A\in \mathbb{R}^{m\times n},\qquad b\in \mathbb{R}^{m}
\end{equation}
 by an $n\times m$ rectangular matrix $B$ from the right or the left, and using the generalized minimal residual (GMRES) method \cite{saad1986} for solving the preconditioned least squares problems (AB-GMRES and BA-GMRES methods, respectively).
For ill-conditioned problems, AB-GMRES and BA-GMRES were shown to be more robust compared to the preconditioned CGNE and CGLS, respectively. Note here that BA-GMRES works with Krylov subspaces in $n$-dimensional space, whereas AB-GMRES works with Krylov subspaces in $m$-dimensional space. \par
The algorithm for BA-GMRES is given as follows. 
\begin{algorithm}
\caption{BA-GMRES}
\label{BA-GMRES method}
\begin{algorithmic}[1]
\State Choose $x_0\in \mathbb{R}^{n}$,\quad$r_0=b-Ax_0$,\quad, $w_0=Br_0$,\quad $\boldsymbol{v_1}=w_0/\|w_0\|_2$,
\For{$i=1,2,\dots,k$}
\State $w_i=BAv_i$,
\For{$j=1,2,\dots,i$}
\State $h_{i,j}=w_i^{{\top}}v_j$, \quad  $w_i=w_i-h_{j,i}v_j$,
\EndFor
\State $h_{i+1,i}=\|w_i\|_2$,  \quad $v_{i+1}=w_i/h_{i+1,i}$,
\State Compute $y_i\in \mathbb{R}^i$ which minimizes $\|w_i\|_2=\|\|w_0\|_2e_1-H_{i+1,i}y_i\|_2$,
\State $x_i=x_0+[\boldsymbol{v_1}, \boldsymbol{v_2}, \dots, v_i]y_i$, \qquad $r_i=b-Ax_i$.
\If{$\|A^{\top}r_i\|_2 < \epsilon\|A^{\top}r_0\|_2$}
\State stop
\EndIf
\EndFor
\end{algorithmic}
\label{ALbagmres}
\end{algorithm}
\par
The BA-GMRES method \cite{Hayami10}, applies GMRES to
\begin{equation}\label{BAGMRESEQ}
  BA\boldsymbol{x}=B\boldsymbol{b}, \qquad A\in \mathbb{R}^{m\times n},\qquad B\in \mathbb{R}^{n\times m},\qquad b\in \mathbb{R}^{m},
\end{equation}
and is equivalent to the original least squares problem $(\ref{eqstar})$
if and only if $\mathcal{R}(B^{\top}BA)=\mathcal{R}(A)$. \par
If we let $B=A^{\top}$, we have the normal equations
\begin{equation}\label{thesiseq415}
   A^{\top}A\boldsymbol{x}=A^{\top}\boldsymbol{b}.
\end{equation}
\par
One can precondition this system by an explicit matrix $P\in\mathbb{R}^{n\times n},$ which is given by
\begin{equation*}
    PA^{\top}A\boldsymbol{x}=PA^{\top}\boldsymbol{b}.
\end{equation*}
Forming an explicit matrix $P$ requires computation time and storage space, especially when there is a requirement to form the normal equation matrix $A^{\top}A$ explicitly.\par
Applying NR-SOR to the normal equations for $l$ steps, which avoids forming the normal equation matrix $A^{\top}A$ of (\ref{thesiseq415}) explicitly, is mathematically equivalent to providing a preconditioning matrix $P^{(l)}$ such that
\begin{equation*}
P^{(l)}A^{\top}A\boldsymbol{x}=P^{(l)}A^{\top}\boldsymbol{b}.
\end{equation*}\par
Introducing a stationary iteration method inside the GMRES iteration instead of forming an explicit preconditioning matrix to precondition GMRES, gives the inner-iteration preconditioned GMRES \cite{Morikuni2013,Morikuni15}. Morikuni \cite{Morikuni15} presents different stationary iterative methods combined with AB-GMRES and BA-GMRES and compares with other methods.\par
As other earlier work, we mention FGMRES \cite{saad1993flexible}, which is more related to AB-GMRES but applies different preconditioners at each step. SOR was used as an inner-iteration preconditioner with GCR \cite{abe2005variable}, and SOR as an inner-iteration preconditioner with GMRES \cite{delong1997sor, delong1998sor}.\par
Splitting the matrix $A^\top A$ into three parts gives
\begin{equation}
    A^\top A=D+L+U=M-N,
\end{equation}
where $D$ is the diagonal part, $L$ is the strictly lower triangular part, and $U$  is the strictly upper triangular part, respectively.
The NR-SOR method chooses
\begin{equation}
    M=\frac{1}{\omega}D+L,\qquad N=(\frac{1}{\omega}-1)D-U.
\end{equation}  
NR-SOR is equivalent to applying SOR to the normal equation
\begin{equation}
    A^{\mathsf{T}}A\boldsymbol{x}=A^{\mathsf{T}}\boldsymbol{b}
\end{equation}
as
\begin{equation}
    (M-N)\boldsymbol{x}=A^{\mathsf{T}}\boldsymbol{b}.
\end{equation}
The algorithm is given below.\par
Let $a_i$ be the $i$th column of $A$, $i=1,2,\dots,n$. 
Suppose $a_i\neq 0, i=1,2, \dots,n.$
\begin{algorithm}
\caption{NR-SOR}
\begin{algorithmic}[1]
\State Let $x^0$ be the initial solution and $r=b-Ax^0$, $0<\omega<2$.
\For{$k=1,2,\dots,l$}
\For{$i=1,2,\dots,n$}
\State $\delta_i=\omega(r,a_i)/||a_i||_2^2$,
\State $x^{k+1}_i=x^k_i+\delta_i$,
\State $r=r-\delta_ia_i$.
\EndFor
\EndFor
\end{algorithmic}
\label{ALNRSOR}
\end{algorithm}
\par
Note that, the computation of $||a_i||_2^2$ is done only once in the beginning.\par
Using NR-SOR as an inner-iteration preconditioner is a way of implicit preconditioning, but has an explicit form for theoretical analysis. In the NR-SOR, let $A^{\top}A=M-N$, when $M$ is nonsingular. Then, 
\begin{align*}
     P^{(l)}A^{\top}A&=(\sum_{i=1}^{l-1}(M^{-1}N)^i+\rm I\it)M^{-1}A^{\top}A\\&=(\sum_{i=1}^{l-1}(M^{-1}N)^i+\rm I\it)M^{-1}(M-N)\\&=(\sum_{i=1}^{l-1}(M^{-1}N)^i+\rm I\it)(\rm I\it-M^{-1}N)\\&=\sum_{i=1}^{l-1}(M^{-1}N)^i+\rm I\it-\sum_{i=1}^{l}(M^{-1}N)^i\\&=\rm I\it-(M^{-1}N)^l
     \\&=\rm I\it-H^l
\end{align*}
where $H=M^{-1}N$. Hence, the eigenvectors of $P^{(l)}A^{\top}A$ remain the same for $l=1, 2, \dots.$\par
The algorithm for using NR-SOR as inner-iteration preconditioner in BA-GMRES is as follows \cite{Morikuni2013,Morikuni15}.
\begin{algorithm}
\caption{NR-SOR inner-iteration BA-GMRES}
\label{NR-SOR inner-iteration BA-GMRES}
\begin{algorithmic}[1]
\State Choose $x_0\in \mathbb{R}^{n}$,\quad$r_0=b-Ax_0$,
\State apply $l$ steps SOR to $A^{\top}Aw=A^{\top}r_0$ to obtain $w_0=P^{l}A^{\top}r_0$, (NR-SOR),
\State $\boldsymbol{v_1}=w_0/\|w_0\|_2$,
\For{$i=1,2,\dots,k$}
\State $u_i=Av_i$,
\State apply $l$ steps SOR to $A^{\top}Aw=A^{\top}u_i$ to obtain $w_i=P^{l}A^{\top}u_i$, (NR-SOR),
\For{$j=1,2,\dots,i$}
\State $h_{i,j}=w_i^{{\top}}v_j$, \quad  $w_i=w_i-h_{j,i}v_j$,
\EndFor
\State $h_{i+1,i}=\|w_i\|_2$,  \quad $v_{i+1}=w_i/h_{i+1,i}$,
\State Compute $y_i\in \mathbb{R}^i$ which minimizes $\|w_i\|_2=\|\|w_0\|_2e_1-H_{i+1,i}y_i\|_2$,
\State $x_i=x_0+[\boldsymbol{v_1}, \boldsymbol{v_2}, \dots, v_i]y_i$, \qquad $r_i=b-Ax_i$.
\If{$\|A^{\top}r_i\|_2 < \epsilon\|A^{\top}r_0\|_2$}
\State stop
\EndIf
\EndFor
\end{algorithmic}
\label{ALnrsorbagmres}
\end{algorithm}
\subsection{Eigenvalue distribution for inner-iteration preconditioning}
We use a test matrix \cite{Morikuni15} to explain our analysis. Let
\begin{equation}\label{testeq450}
   A= U\left(
  \begin{array}{ccccccc}
    1 &1  &  &  &  &  & 0 \\
     &  & 0.9 &0.9  &  &  &  \\
     &  &  &  &\ddots  & \ddots &  \\
     &  &  &  &  &0.1  &0.1  \\
    0 &  &  &  &  &  &  \\
  \end{array}
\right) 
V^{\top}\in \mathbb{R}^{100\times 20}.
\end{equation}
where $U\in \mathbb{R}^{100\times 100}$ and $V\in\mathbb{R}^{20\times 20}$ are orthogonal matrices computed with the QR factorization of random
matrices. Thus, $A$ is rank-deficient, with rank 10. \par
The residual bound given in \cite{Morikuni15}, represented by the exponent of the spectral radius after taking the logarithm, closely resembles a straight line, providing a pessimistic estimate of the actual residual and failing to adequately explain the observed superlinear convergence .
\par
In order to analyze the convergence of BA-GMRES, in (\ref{eqfirst1}), let $A=B^{(l)}A=P^{(l)}A^{\top}A,\; \boldsymbol{b}=B^{(l)}b=P^{(l)}A^{\top}b.$ 
Figure \ref{cp4eigA} shows the nonzero singular values of $A$. Figure \ref{cp4eigATA} shows the nonzero eigenvalue of $A^{\top}A$. Figure \ref{cp4eigHcp4eigH} shows the nonzero eigenvalues of $H=M^{-1}N$. Figure \ref{cp4l4cp4l4} shows the eigenvalues of $B^{(l)}A=P^{(l)}A^{\top}A=\rm I\it-H^l,l=\rm4$. Figure \ref{cp4l8cp4l8} shows the eigenvalues of $B^{(l)}A=P^{(l)}A^{\top}A=\rm I\it-H^l,l=\rm8$.  Table \ref{tbeigen10020} gives the values for the above figures.
\par
\begin{figure}
  \centering
  \begin{tabular}{c}
  \includegraphics[width=12cm]{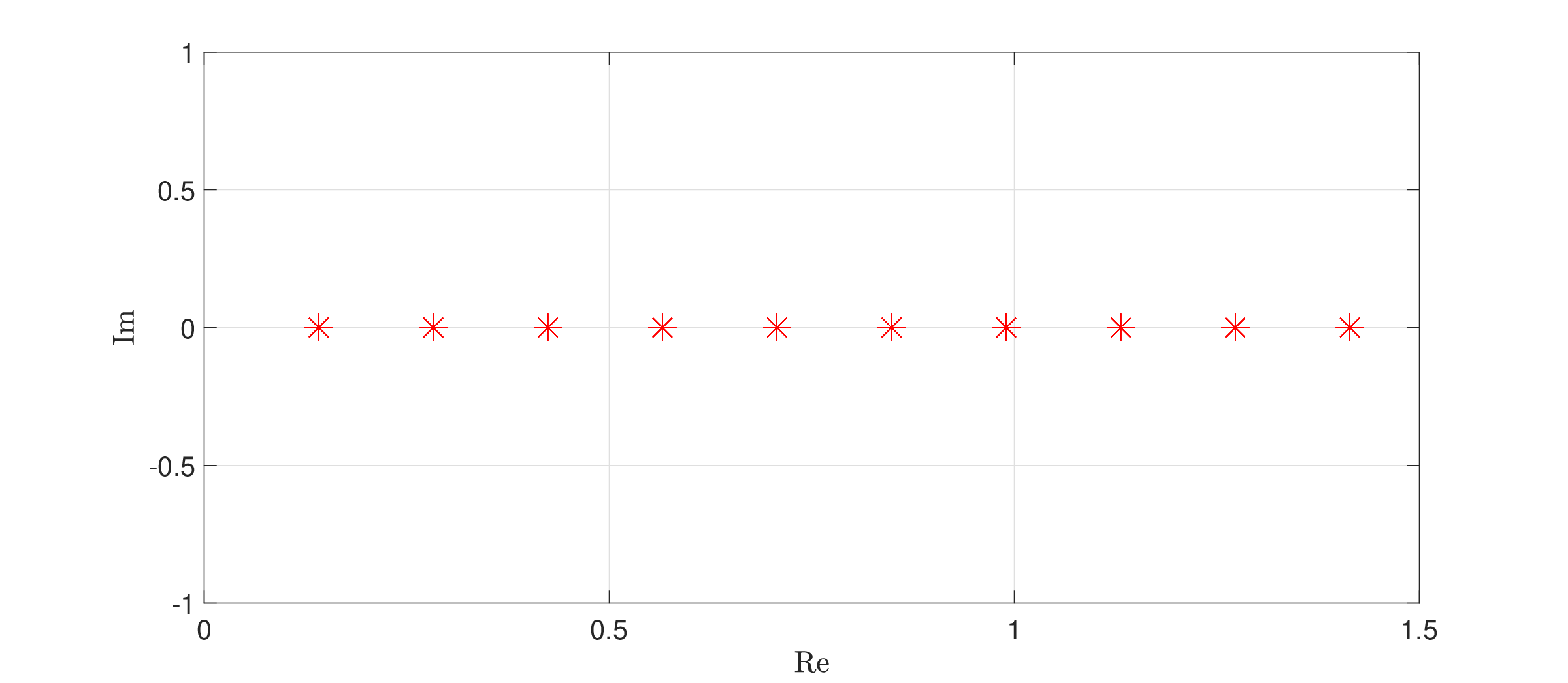}
\end{tabular}
\caption{The nonzero singular values of the test matrix $A$.}
  \label{cp4eigA}
 \end{figure}
\par
\begin{figure}
  \centering
  \begin{tabular}{c}
  \includegraphics[width=12cm]{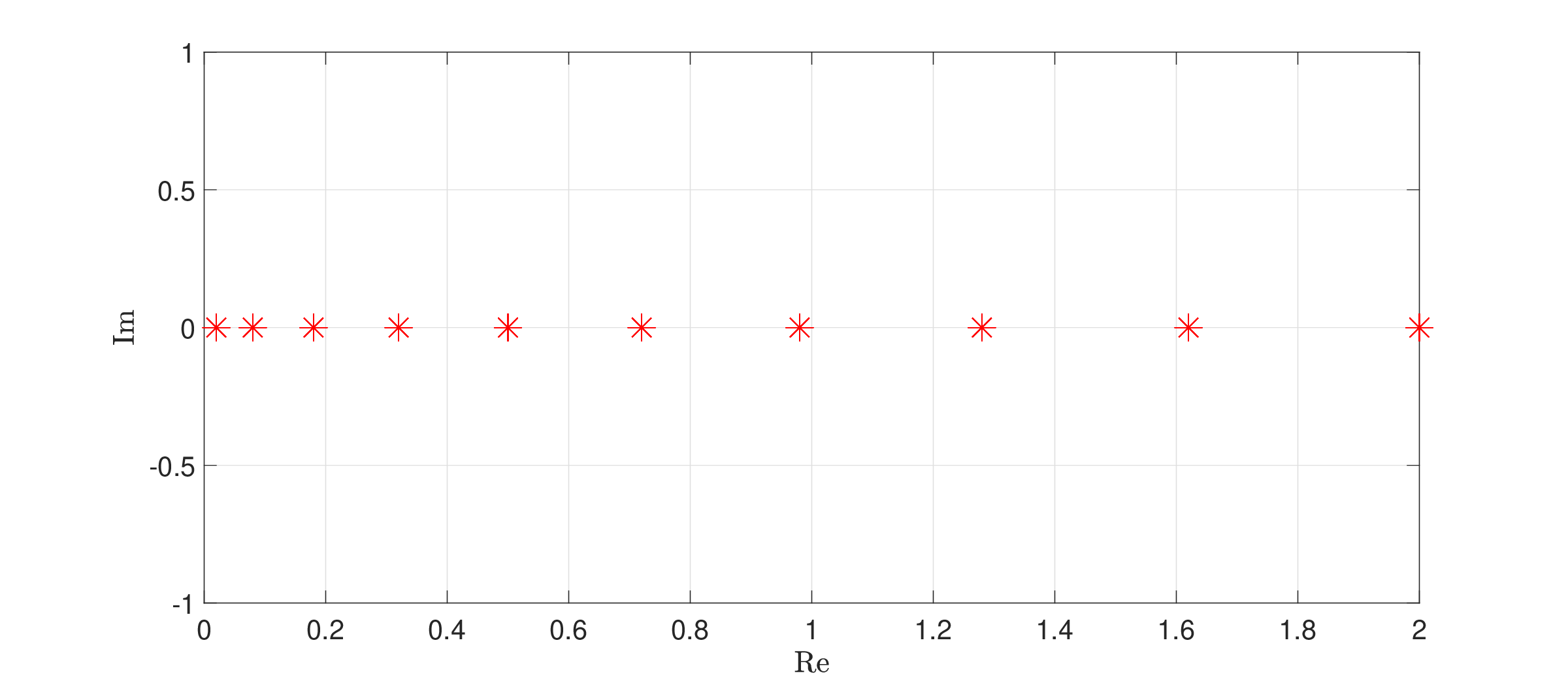}
\end{tabular}
  \caption{The nonzero eigenvalues of the normal equation matrix $A^\top A$ of the test matrix $A$.}
  \label{cp4eigATA}
 \end{figure}
\par
\begin{figure}
  \centering
  \begin{tabular}{c}
  \includegraphics[width=12cm]{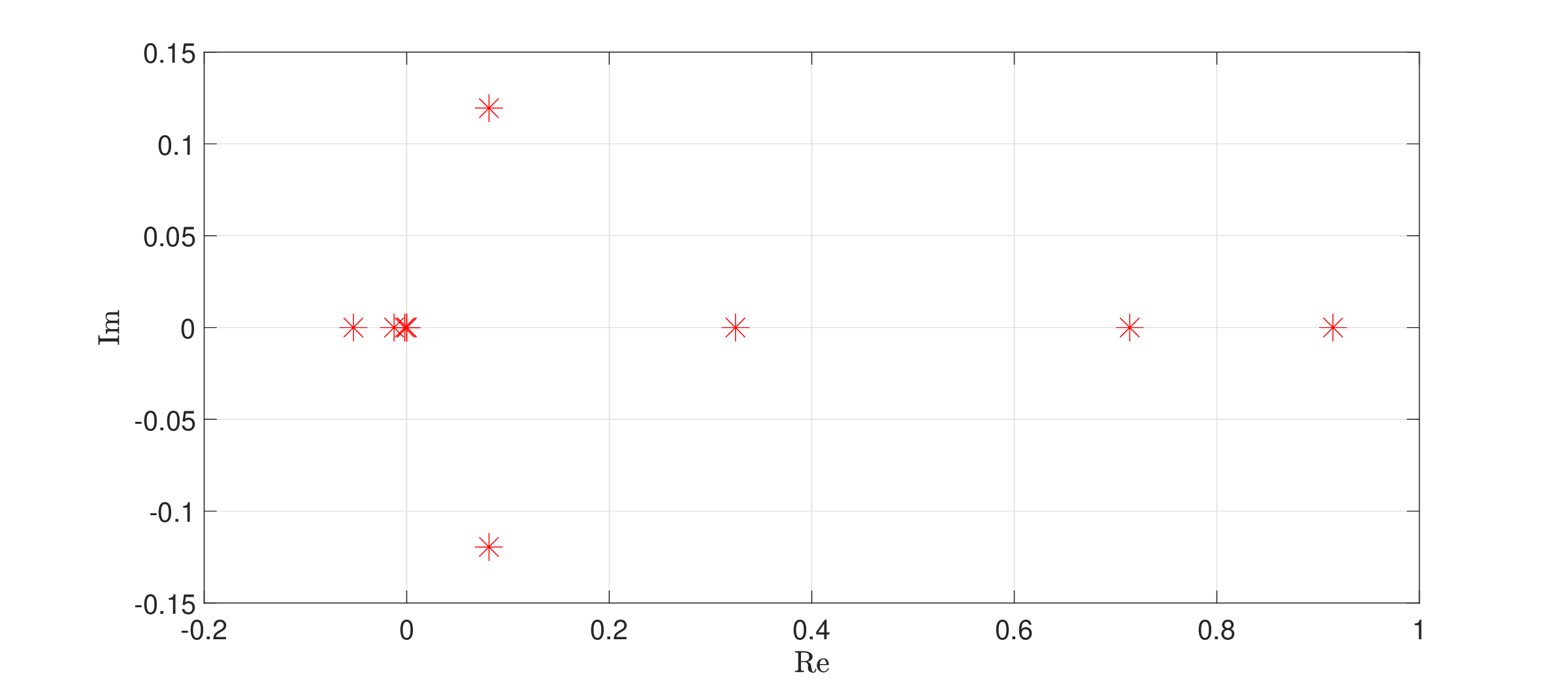}
\end{tabular}
  \caption{The nonzero eigenvalues of $H=M^{-1}N$ of the test matrix $A$.}
  \label{cp4eigHcp4eigH}
 \end{figure}
\par
\begin{figure}
  \centering
  \begin{tabular}{c}
  \includegraphics[width=12cm]{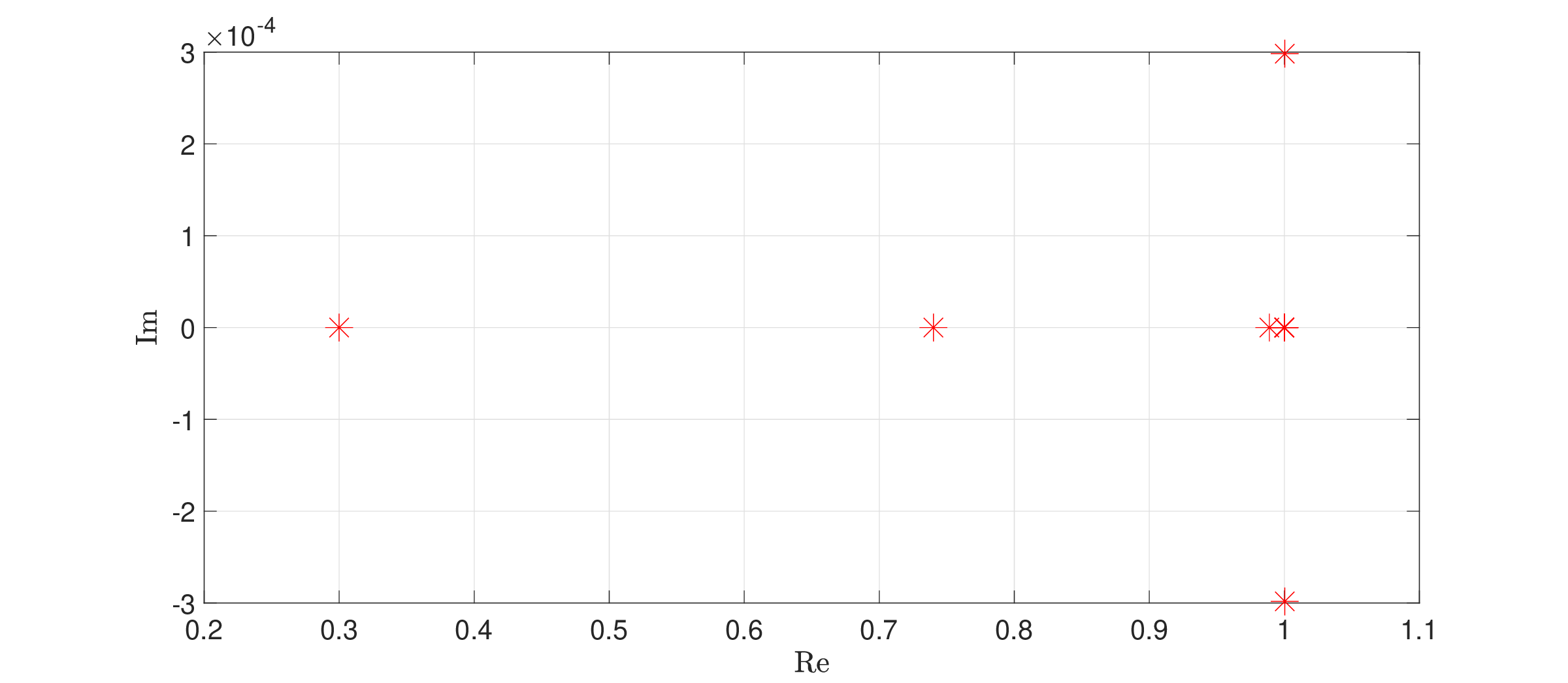}
\end{tabular}
  \caption{The nonzero eigenvalues of $B^{(l)}A=\rm I\it-H^l(l=\rm4)$ of the test matrix $A$.}
  \label{cp4l4cp4l4}
 \end{figure}
\par
\begin{figure}
  \centering
  \begin{tabular}{c}
  \includegraphics[width=12cm]{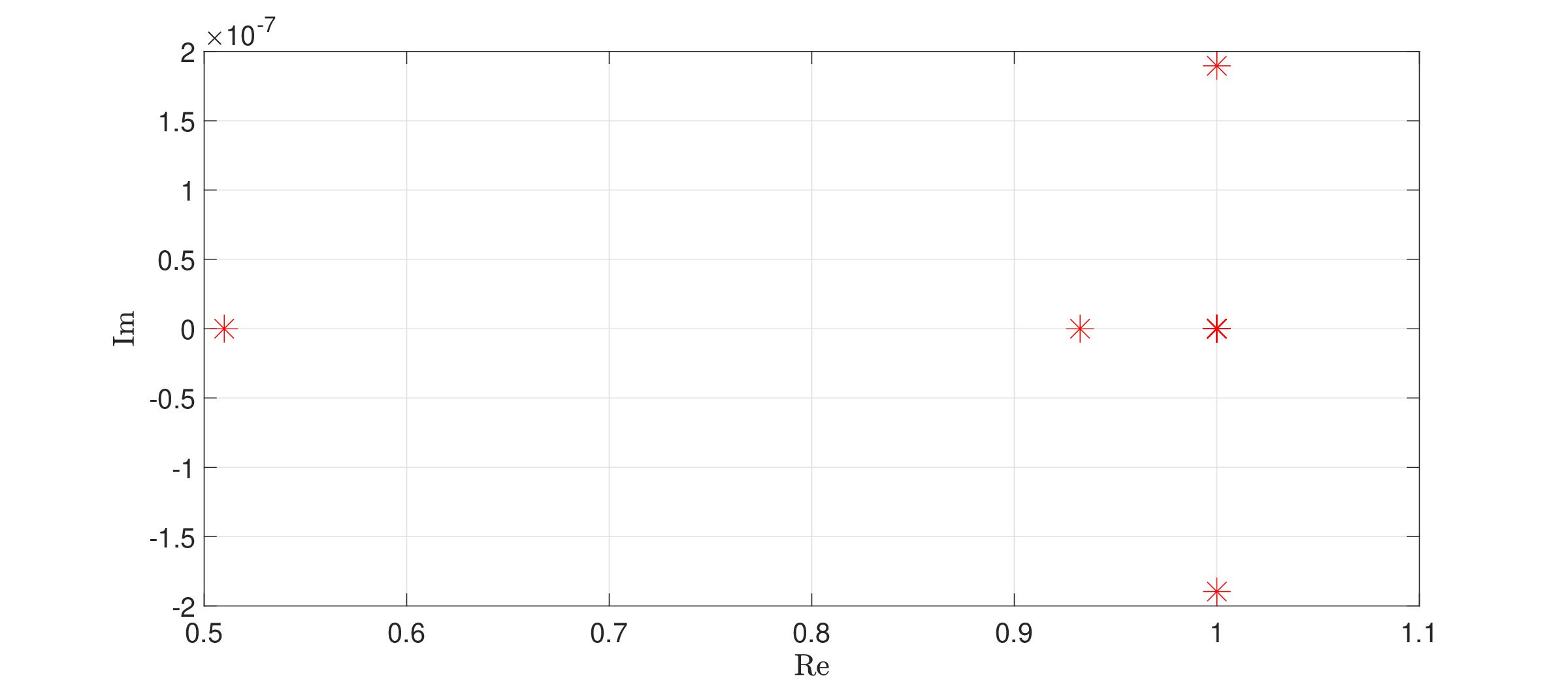}
\end{tabular}
  \caption{The nonzero eigenvalues of $B^{(l)}A=\rm I\it-H^l(l=\rm8)$ of the test matrix $A$.}
  \label{cp4l8cp4l8}
 \end{figure}
\par
\begin{table}[h]
\caption{The singular values of $A$,  eigenvalues of $A^{\top}A$, $H(M^{-1}N)$, and $B^{(l)}A=\rm I\it-H^l(l=\rm4, 8)$.}
    \centering
\begin{tabular}{c|rrrrr}
  \multicolumn{1}{c|}{} & \multicolumn{1}{c}{$A$}&\multicolumn{1}{c}{$A^{\top}A$}&$H=M^{-1}N$ & $B^{(l)}A=\rm I\it-H^l(l=\rm4)$& $B^{(l)}A=\rm I\it-H^l(l=\rm8)$\\
  \hline\rule{0pt}{12pt}
 1 & 1.41&  2.00& 0.00& 1.00& 1.00\\
 2 & 1.27&  1.62 & 0.00 & 1.00&1.00 \\
 3 & 1.31&  1.28& 0.00& 1.00&1.00 \\
  4 & 0.99&  0.98 &0.01& 1.00& 1.00\\
   5 & 0.85&  0.72&0.05 &1.00 &1.00 \\
   6 & 0.71&  0.50 &$0.08+0.12i$&$1.00+2.98\times 10^{-4}i$ &$1.00+1.90\times 10^{-7}i$ \\
    7 & 0.57&  0.32 &$0.08-0.12i$ &$1.00-2.98\times 10^{-4}i$& $1.00-1.90\times 10^{-7}i$\\
    8 & 0.42&  0.18&0.32&0.99 & 1.00 \\
    9 & 0.28&  0.08& 0.71&0.74 &0.93 \\
    10 & 0.14&  0.02&0.91 &0.30 &0.51 \\
\end{tabular}
 
    \label{tbeigen10020}
\end{table}

\subsection{Estimation for the test matrix}
As for the test matrix of (\ref{testeq450}), $A\in \mathbb{R}^{100\times 20}$, $A=B^{(l)}A=P^{l}A^{\top}A=\rm I\it-H^l$ has only one cluster of eigenvalues around the center 1, and the others are separate eigenvalues as shown in Figure \ref{cp4l4cp4l4} and \ref{cp4l8cp4l8} for $l=\rm 4, 8$. 
\begin{table}[h]
\caption{The eigenvalue distribution of $A=B^{(l)}A=\rm I\it-H^l(l=\rm 8)$}
\begin{center}
\begin{tabular}{r|rr}

  \multicolumn{1}{c|}{eigenvalues} & \multicolumn{1}{c}{structure}&\multicolumn{1}{c}{value} \\
  \hline\rule{0pt}{12pt}
$\lambda_1$ & $1+\epsilon_1$&  $1+8.00\times 10^{-15}$ \\
$\lambda_2$ & $1+\epsilon_2$&  $1+3.11\times 10^{-15}$  \\
$\lambda_3$ & $1+\epsilon_3$&  $1+2.44\times 10^{-15}$ \\
$\lambda_4$ & $1+\epsilon_4$&  $1+5.86\times 10^{-11}$ \\
$\lambda_5$ & 1&  1 \\
$\lambda_6$ &$\lambda_6$ & $1.00+1.90\times 10^{-7}i$  \\
$\lambda_7$ &$\lambda_7$&  $1.00-1.90\times 10^{-7}i$  \\
$\lambda_8$&$\lambda_8$ & 0.9999   \\
$\lambda_9$ &$\lambda_9$&  0.9325  \\
$\lambda_{10}$ &$\lambda_{10}$ & 0.5099  \\
\end{tabular}
\end{center}
\label{tbsetting}
\end{table}
Thus, according to Table \ref{tbsetting} where $d=10$ in (\ref{eqLambda}), we have as in (\ref{euqationdef61}) with $k=d=10$
\begin{equation*}\label{eqLambda1}
    \Lambda_{\epsilon}\equiv\Lambda_{10}^{10}=\left(
  \begin{array}{cccc}
   1+\epsilon_1  &(1+\epsilon_1)^2 & \cdots &(1+\epsilon_1)^{10}  \\
   1+\epsilon_2  &(1+\epsilon_2)^2 & \cdots &(1+\epsilon_2)^{10}  \\
   1+\epsilon_3  &(1+\epsilon_3)^2 & \cdots &(1+\epsilon_3)^{10}  \\
   1+\epsilon_4  &(1+\epsilon_4)^2 & \cdots &(1+\epsilon_4)^{10}  \\
  1&1 & \cdots &1\\
   \lambda_6  &\lambda_6^2 & \cdots &\lambda_6^{10}  \\
   \lambda_7  & \lambda_7^2 & \cdots &\lambda_7^{10}  \\
    \lambda_8  &\lambda_8^2 & \cdots &\lambda_8^{10}  \\
     \lambda_9  &\lambda_9^2 & \cdots &\lambda_9^{10}  \\
    \lambda_{10} & \lambda_{10}^2 & \cdots  & \lambda_{10}^{10} \\
  \end{array}
\right)\in \mathbb{C}^{10\times 10}.
\end{equation*}
For step $k<d$,

\begin{equation*}\label{eqLambda1}
    \Lambda_\epsilon=\left(
  \begin{array}{cccc}
   1+\epsilon_1  &(1+\epsilon_1)^2 & \cdots &(1+\epsilon_1)^{k}  \\
   1+\epsilon_2  &(1+\epsilon_2)^2 & \cdots &(1+\epsilon_2)^{k}  \\
   1+\epsilon_3  &(1+\epsilon_3)^2 & \cdots &(1+\epsilon_3)^{k}  \\
   1+\epsilon_4  &(1+\epsilon_4)^2 & \cdots &(1+\epsilon_4)^{k}  \\
  1&1 & \cdots &1\\
   \lambda_6 &\lambda_6^2 & \cdots &\lambda_6^k \\
   \lambda_7  & \lambda_7^2 & \cdots &\lambda_7^{k}  \\
    \lambda_8  &\lambda_8^2 & \cdots &\lambda_8^{k}  \\
     \lambda_9  &\lambda_9^2 & \cdots &\lambda_9^{k}  \\
    \lambda_{10} & \lambda_{10}^2 & \cdots  & \lambda_{10}^{k} \\
  \end{array}
\right)\in \mathbb{C}^{10\times k}.
\end{equation*}
Since, $\epsilon=\max_k|\epsilon_k|<10^{-10}\thinspace (k=1,2,3,4)$, which is tiny,

\begin{equation*}\label{eqLambda2}
    \Lambda_\epsilon\approx\widetilde{\Lambda_\epsilon}=\left(
  \begin{array}{cccc}
   1+\epsilon_1  &1+2\epsilon_1 & \cdots &1+k\epsilon_1  \\
   1+\epsilon_2  &1+2\epsilon_2 & \cdots &1+k\epsilon_2  \\
   1+\epsilon_3  &1+2\epsilon_3 & \cdots &1+k\epsilon_3  \\
   1+\epsilon_4  &1+2\epsilon_4 & \cdots &1+k\epsilon_4  \\
  1&1 & \cdots &1\\
   \lambda_6  &\lambda_6^2 & \cdots &\lambda_6^k \\
   \lambda_7  & \lambda_7^2 & \cdots &\lambda_7^{k}  \\
    \lambda_8  &\lambda_8^2 & \cdots &\lambda_8^{k}  \\
     \lambda_9  &\lambda_9^2 & \cdots &\lambda_9^{k}  \\
    \lambda_{10} & \lambda_{10}^2 & \cdots  & \lambda_{10}^{k} \\
  \end{array}
\right)\in \mathbb{C}^{10\times k}.
\end{equation*}
Seperating $\widetilde{\Lambda_\epsilon}$ into two matrices, $\widetilde{\Lambda_\epsilon}=\Lambda_s+P,$ where

\begin{equation*}\label{eqLambda3}
    \Lambda_s=\left(
  \begin{array}{cccc}
  1&1 & \cdots &1\\
  1&1 & \cdots &1\\
  1&1 & \cdots &1\\
    1&1 & \cdots &1\\
  1&1 & \cdots &1\\
   \lambda_6  &\lambda_6^2 & \cdots &\lambda_6^k \\
   \lambda_7  & \lambda_7^2 & \cdots &\lambda_7^{k}  \\
    \lambda_8  &\lambda_8^2 & \cdots &\lambda_8^{k}  \\
     \lambda_9  &\lambda_9^2 & \cdots &\lambda_9^{k}  \\
    \lambda_{10} & \lambda_{10}^2 & \cdots  & \lambda_{10}^{k} \\
  \end{array}
\right)\in \mathbb{C}^{10\times k},\qquad P=\left(
  \begin{array}{cccc}
   \epsilon_1  &2\epsilon_1 & \cdots &k\epsilon_1  \\
   \epsilon_2  &2\epsilon_2 & \cdots &k\epsilon_2  \\
   \epsilon_3  &2\epsilon_3 & \cdots &k\epsilon_3  \\
   \epsilon_4  &2\epsilon_4 & \cdots &k\epsilon_4  \\
  0  &0 & \cdots &0 \\
   0  &0 & \cdots &0 \\
   0  &0 & \cdots &0 \\
    0  &0 & \cdots &0 \\
     0  &0 & \cdots &0 \\
     0  &0 & \cdots &0 \\
  \end{array}
\right)\in \mathbb{C}^{10\times k}.
\end{equation*}

\begin{equation*}\label{eqLambda5}
    \widetilde{\Lambda_s}=\left(
  \begin{array}{cccc}
  1&1 & \cdots &1\\
   \lambda_6  &\lambda_6^2 & \cdots &\lambda_6^k \\
   \lambda_7  & \lambda_7^2 & \cdots &\lambda_7^{k}  \\
    \lambda_8  &\lambda_8^2 & \cdots &\lambda_8^{k}  \\
     \lambda_9  &\lambda_9^2 & \cdots &\lambda_9^{k}  \\
    \lambda_{10} & \lambda_{10}^2 & \cdots  & \lambda_{6}^{k} \\
  \end{array}
\right)\in \mathbb{C}^{6\times k}.
\end{equation*}\par
For $k=6$, we have
\begin{equation*}
    \text{det} \widetilde{\Lambda_{\it s}}=\prod_{6\leq i<j\leq 10}(\lambda_i-\lambda_j)\prod_{i=6}^{10}(\lambda_i-1).
\end{equation*}\par
Since $\lambda_i\neq\lambda_j\neq 1\neq 0$, $(6\leq i<j\leq 10)$,
\begin{equation*}
     \text{det} \widetilde{\Lambda_{\it s}}\neq 0 \Rightarrow \text{rank} \Lambda_s=6\quad (k=6).
\end{equation*}\par
Note 
\begin{equation*}
    \|V_d\text{diag}[c_1, c_2, \cdots, c_d](\Lambda_sy-[1,1,\cdots,1]^{\top})\|_2\leq\|V_d\text{diag}[c_1, c_2, \cdots, c_d]\|_2\|\Lambda_s y-[1,1,\cdots,1]^{\top}\|_2
\end{equation*}
holds in general.\par
$\|\widetilde{\Lambda_s}y-[1,1,\cdots,1]^{\top}\|_2=0$ and $\|\Lambda_s y-[1,1,\cdots,1]^{\top}\|_2=0$ share the same solution $y$ if $k=6$. \par
Note, rank$\Lambda_s\leq 6$ for $k\leq 6.$ rank$\Lambda_s=k\quad(1\leq k\leq 6)$,  
 rank$\Lambda_s=6\quad (k>6)$, if $\lambda_i\neq\lambda_j\neq 1\neq 0$, $(6\leq i<j\leq 10)$.\par
Let $y_1 =\arg \min_{y_\in \mathcal{R}^k} \|\widetilde{\Lambda_s} y_k-[1, 1, \dots, 1]^{\top}\|_2$\par
Note
 \begin{align*}
    \min_y \|\Lambda_\epsilon y-[1, 1, \dots, 1]^{\top}\|_2&\leq \|\Lambda_\epsilon y_1-[1, 1, \dots, 1]^{\top}\|_2\\ &\approx\|\Lambda_s y_1-[1, 1, \dots, 1]^{\top}+Py_1\|_2 \\&\leq \|\Lambda_s y_1-[1, 1, \dots, 1]^{\top}\|_2+\|Py_1\|_2\\ &= \|\widetilde{\Lambda_s} y_1-[1, 1, \dots, 1]^{\top}\|_2+\|Py_1\|_2 \quad (k=6)
    \\&= \|Py_1\|_2 \quad (k=6)
 \end{align*} 
Notice that when $k=6$, $\|\widetilde{\Lambda_s} y_1-[1, 1, \dots, 1]^{\top}\|_2=0.$
\begin{equation*}
    \|\Lambda_\epsilon y-[1, 1, \dots, 1]^{\top}\|_2\leq \|Py_1\|_2,
\end{equation*}
where $y_1=(y_1^1, y_1^2, \dots, y_1^6)^{\top}$, and
\begin{equation*}
    Py_1=\left(
  \begin{array}{c}
   \epsilon_1y_1^1+2\epsilon_1y_1^2+\dots+ 6\epsilon_1y_1^6  \\
 \epsilon_2y_1^1+2\epsilon_2y_1^2+\dots+ 6\epsilon_2y_1^6  \\
   \epsilon_3y_1^1+2\epsilon_3y_1^2+\dots+ 6\epsilon_3y_1^6  \\
 \epsilon_4y_1^1+2\epsilon_1y_4^2+\dots+ 6\epsilon_4y_1^6  \\
  0\\
  0\\
   0\\
    0\\
     0\\
    0\\
  \end{array}
\right).
\end{equation*}
Since

\begin{equation*}
    \|\widetilde{\Lambda_s} y_1-[1, 1, \dots, 1]^{\top}\|_2=0.
\end{equation*}
\begin{equation*}\label{squreroot}
    \left(
  \begin{array}{cccc}
  1&1 & \cdots &1\\
   \lambda_6  &\lambda_6^2 & \cdots &\lambda_6^6 \\
   \lambda_7  & \lambda_7^2 & \cdots &\lambda_7^{6}  \\
    \lambda_8  &\lambda_8^2 & \cdots &\lambda_8^{6}  \\
     \lambda_9  &\lambda_9^2 & \cdots &\lambda_9^{6}  \\
    \lambda_{10} & \lambda_{10}^2 & \cdots  & \lambda_{10}^{6} \\
  \end{array}
\right)\left(
  \begin{array}{c}
  y_1^1\\
     y_1^2\\
      y_1^3\\
      y_1^4\\
     y_1^5\\ 
       y_1^6\\
  \end{array}
\right)-\left(
  \begin{array}{c}
  1\\
     1\\
      1\\
      1\\
     1\\ 
       1\\
  \end{array}
\right)=\left(
  \begin{array}{c}
  0\\
     0\\
      0\\
      0\\
     0\\ 
       0\\
  \end{array}
\right),
\end{equation*}
    which means $1, \lambda_6, \lambda_7, \lambda_8, \lambda_9$ and $\lambda_{10}$ are roots of
    \begin{equation*}
        f(\gamma)=y_{1}^{6}\gamma^{6}+y_{1}^{5}\gamma^{5}+y_{1}^{4}\gamma^{4}+y_{1}^{3}\gamma^{3}+y_{1}^{2}\gamma^{2}+y_{1}^{1}\gamma-1=0.
    \end{equation*}
    Thus,
    \begin{equation*}
        f(\gamma)=-\frac{1}{\lambda_6\lambda_7\lambda_8\lambda_9\lambda_{10}}(\gamma-1)(\gamma-\lambda_6)(\gamma-\lambda_7)(\gamma-\lambda_8)(\gamma-\lambda_9)(\gamma-\lambda_{10})
    \end{equation*}
    and
      \begin{equation*}
        f'(1)=-\frac{1}{\lambda_6\lambda_7\lambda_8\lambda_9\lambda_{10}}(1-\lambda_6)(1-\lambda_7)(1-\lambda_8)(1-\lambda_9)(1-\lambda_{10}).
    \end{equation*}
    Also, we have
    \begin{equation*} f'(\gamma)=6y_{1}^{6}\gamma^{5}+5y_{1}^{5}\gamma^{4}+4y_{1}^{4}\gamma^{4}+3y_{1}^{3}\gamma^{2}+2y_{1}^{2}\gamma+y_{1}^{1}
    \end{equation*}
    and
     \begin{equation*}
        f'(1)=6y_{1}^{6}+5y_{1}^{5}+4y_{1}^{4}+3y_{1}^{3}+2y_{1}+y_{1}^{1}.
    \end{equation*}
     Let $\epsilon= \max \{\epsilon_1,\epsilon_2,\epsilon_3,\epsilon_4\}<10^{-10}$. Note that
    \begin{align*}
        \|Py_1\|_2&=\|(f'(1)\epsilon_1, f'(1)\epsilon_2 , f'(1)\epsilon_3, f'(1)\epsilon_4)^{\top}\|_2\\&\leq \epsilon\|(f'(1), f'(1), f'(1), f'(1))^{\top}\|_2\\&=2\epsilon|f'(1)|.
    \end{align*}
    Since $\|V_d\text{diag}[c_1, c_2, \dots, c_d]\|_2=2.5068$ \; ($V_d$ is computed by eigenvalue decomposition of $A$. Then, $\boldsymbol{b}$ is decomposed into components parallel to the column vectors $\boldsymbol{v_i} (i=1, 2, \dots, d.)$ of $V_d$ as $b=c_1\boldsymbol{v_1}+c_2\boldsymbol{v_2}+\cdots+c_d\boldsymbol{v_d}$). (\ref{eqVandermonde}) gives the final estimate at the $6$th iteration, as
\begin{align*}
  \| B^{(l)}r_s\|_2&=  \| B^{(l)}(b-Ax_s)\|_2\\
  &\leq \|V_d\text{diag}[c_1, c_2, \dots, c_d]\|_2 \|\Lambda_\epsilon y-[1, 1, \dots, 1]^{\top}\|_2\\&\simeq \|V_d\text{diag}[c_1, c_2, \dots, c_d]\|_2 \|Py_1\|_2\\&<\|V_d\text{diag}[c_1, c_2, \dots, c_d]\|_2 2\epsilon|f'(1)|\\
  &=2\|V_d\text{diag}[c_1, c_2, \dots, c_d]\|_2\epsilon\Big|-\frac{1}{\lambda_6\lambda_7\lambda_8\lambda_9\lambda_{10}}(1-\lambda_6)(1-\lambda_7)(1-\lambda_8)(1-\lambda_9)(1-\lambda_{10})\Big|\\
  &\leq 5.136\times 10^{-10}\times \frac{1}{0.4754}\times0.4901\times0.0675\times0.0001\times (1.8999\times 10^{-7})^2\\
  &=3.49\times 10^{-29}.
\end{align*}
If we choose $\lambda_6$ and $\lambda_7$ to be in the cluster around 1, then $\epsilon<10^{-6}$, and at the $4$th iteration we obtain $\| B^{(l)}r_s\|_2<3.49\times 10^{-12}.$
 \par
\begin{figure}
  \centering
  \begin{tabular}{c}
  \includegraphics[width=12cm]{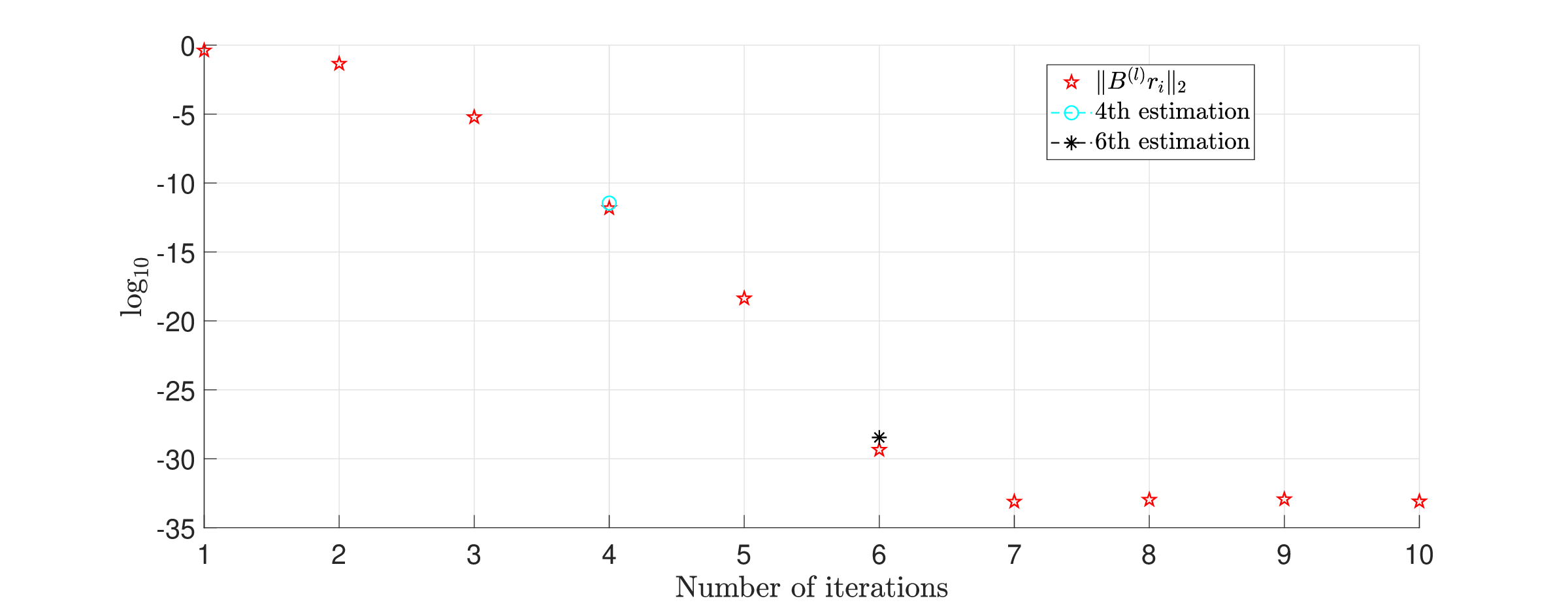}
\end{tabular}
  \caption{$\| B^{(l)}r_s\|_2 (l=8)$ versus the number of iterations for the test matrix $A$ in quadruple precision arithmetic.}
  \label{cp4quadruple}
 \end{figure}
 Figure \ref{cp4quadruple} shows $\| B^{(l)}r_s\|_2$ versus the number of iterations in quadruple precision arithmetic (in double precision arithmetic, the residual stagnation earlier). At the $4$th iteration $\| B^{(l)}r_s\|_2$ is approximately $10^{-12}$, and at the $6$th iteration $\| B^{(l)}r_s\|_2$ is approximately $10^{-29}$, which is close to the estimation. Thus, although $A$ has 10 different singular values, the eigenvalue of the preconditioned matrix $B^{l}A$ is contained in a cluster around 1. Within several steps,
$\| B^{(l)}r_s\|_2$ converges to a tiny level. In other words, the residual norm converges to near zero before the grade $d$.\par
Ipsen's upper bound in \cite{campbell1996gmres} for the non-normal matrix $B^{(l)}A$  gives $\| B^{(l)}r_6\|_2<c\epsilon\|r_0\|_2$, where c is a constant that reflects the distance from separate eigenvalues to the cluster center 1 which is smaller than 0.5 and also reflects the non-normality of $B^{(l)}A$ which is related to $\|V\|_2$, and $\|r_0\|_2=4.55$. Thus, the value of this bound is about $10^{-1}.$ Ipsen's estimation for normal\break matrices in \cite{ipsen2000expressions} gives $\| B^{(l)}r_6\|_2\approx (1/3)\times 0.7^5 \|r_0\|_2\approx0.0560\|r_0\|_2$, which is larger than our estimate. However, $B^{(l)}A$ is non-normal. Our work can be regarded as extending this estimate to the diagonalizable case. Bounds involving\break exponents of the spectral radius after taking $\log$ gives a straight line. Our paper is devoted to illustrating superlinear convergence.\par
Then, we applied our theory to a larger example with 1 step inner-iteration preconditioning to the matrix $A$$=\text{diag((1-1/exp((1:n)/4)))*gallery('orthog',n,1)}$ in Matlab, where $n=1, 001$, and $\text{gallery('orthog', \it n)}$ is a function generates an $n$ dimensional orthogonal matrix, where  we only multiply from the right, so that the matrix is nonsymmetric. The distribution of the eigenvalues after preconditioning is shown in Figure \ref{20241214B}. The corresponding numerical behavior is shown in Figure \ref{2024A}. The bound is pessimistic in the beginning due to the lack of enough centers $\gamma_j$'s to satisfy Theorem \ref{theoremthree}.\par
\begin{figure}
    \centering
    \includegraphics[width=0.7\linewidth]{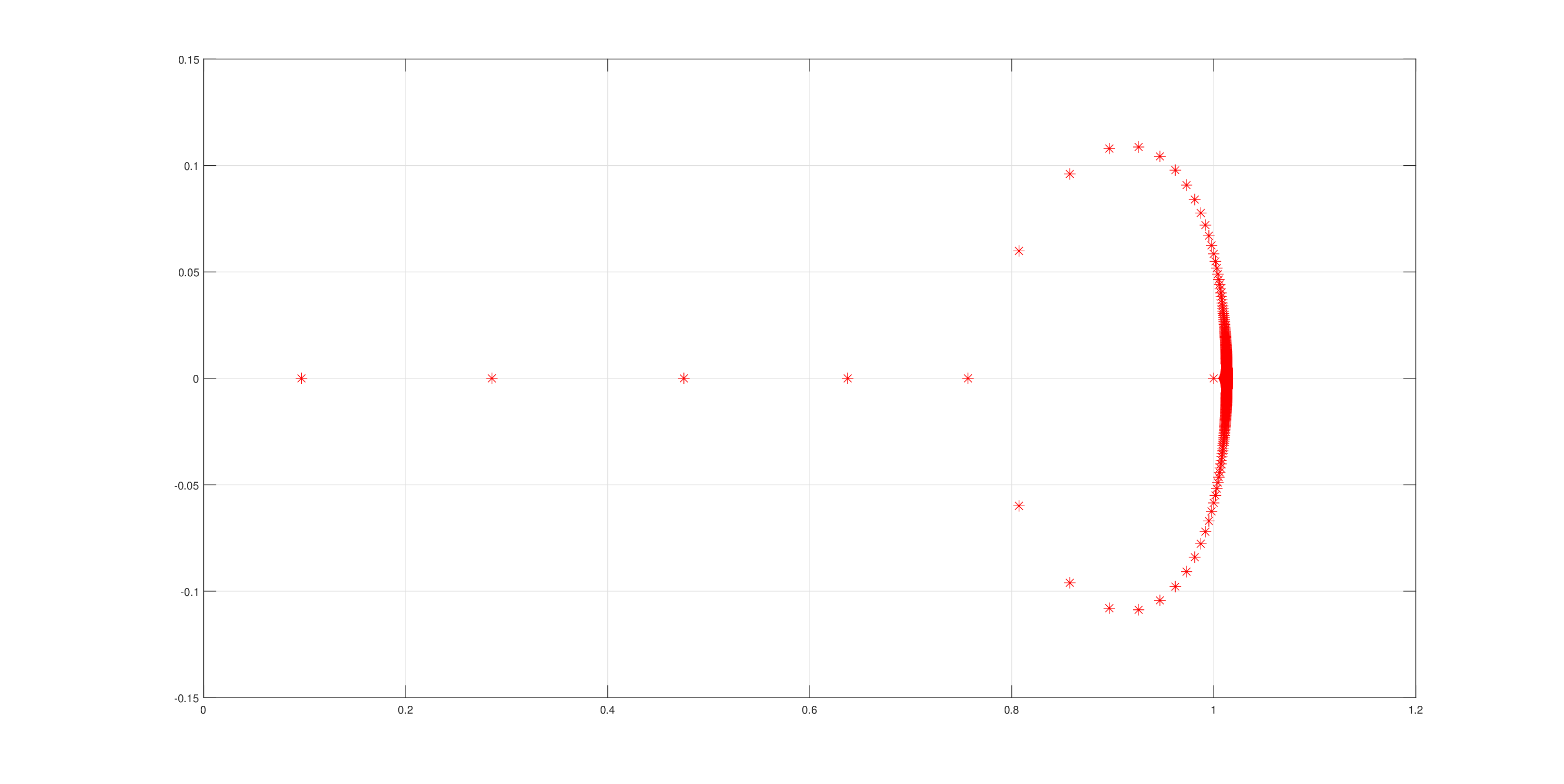}
    \caption{The eigenvalue distribution of $B^{(l)}A$.}
    \label{20241214B}
\end{figure}

\begin{figure}
    \centering
    \includegraphics[width=0.7\linewidth]{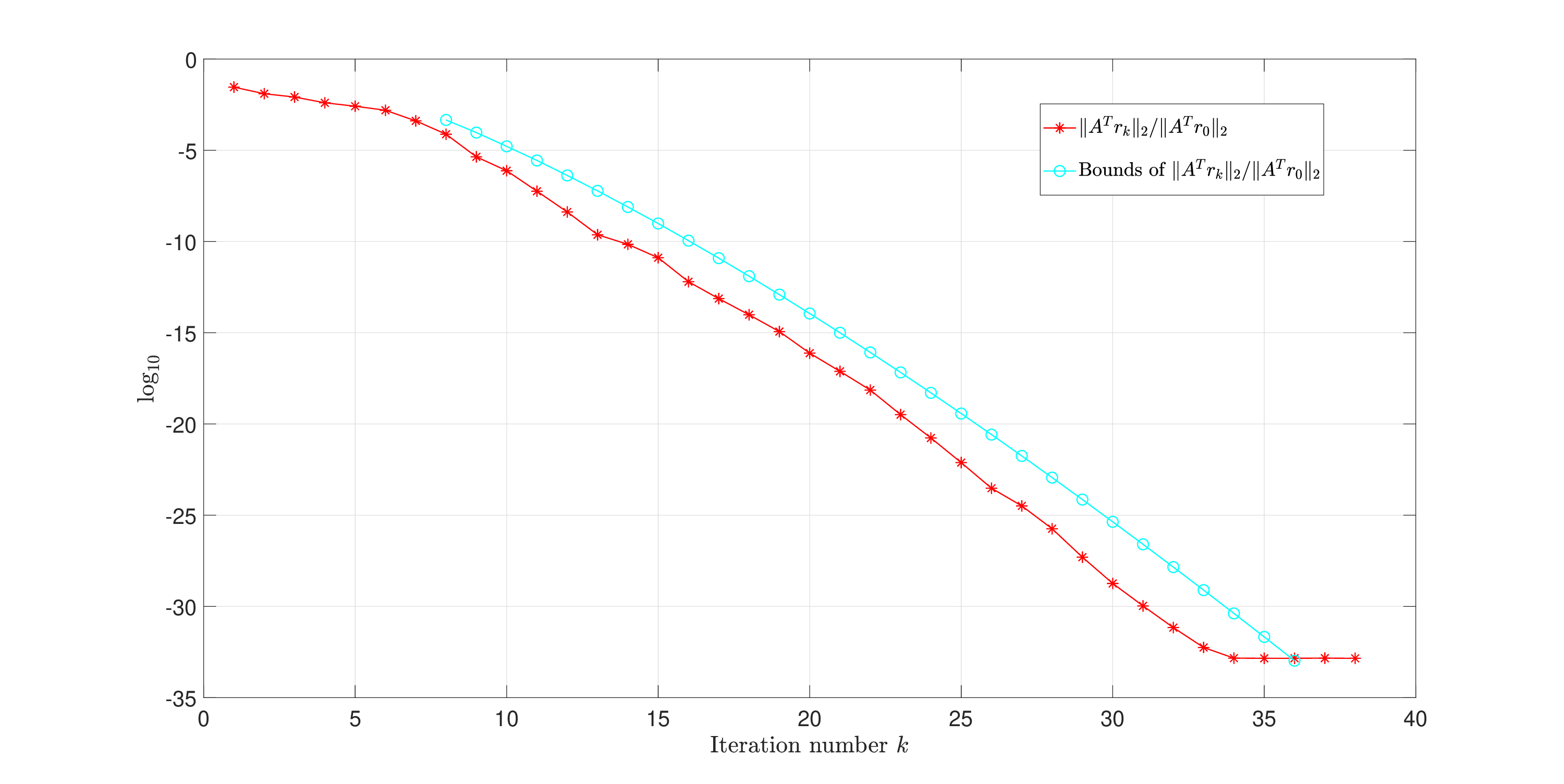}
    \caption{The residual norm and the corresponding residual bound for a larger example.}
    \label{2024A}
\end{figure}
Figure \ref{2024A} shows the case when the eigenvalues are well clustered. The residual converges very fast even though the condition number of $B^{(l)}A$ is $9.18\times 10^{17}$, and the error of the solution is $6.23\times 10^{-14}$. The number of iterations necessary for convergence is much smaller than the matrix size $1, 001$. The bound shows the same tendency as the residual, which shows that our theory can capture the superlinear convergence. \par
\begin{figure}
    \centering
    \includegraphics[width=0.7\linewidth]{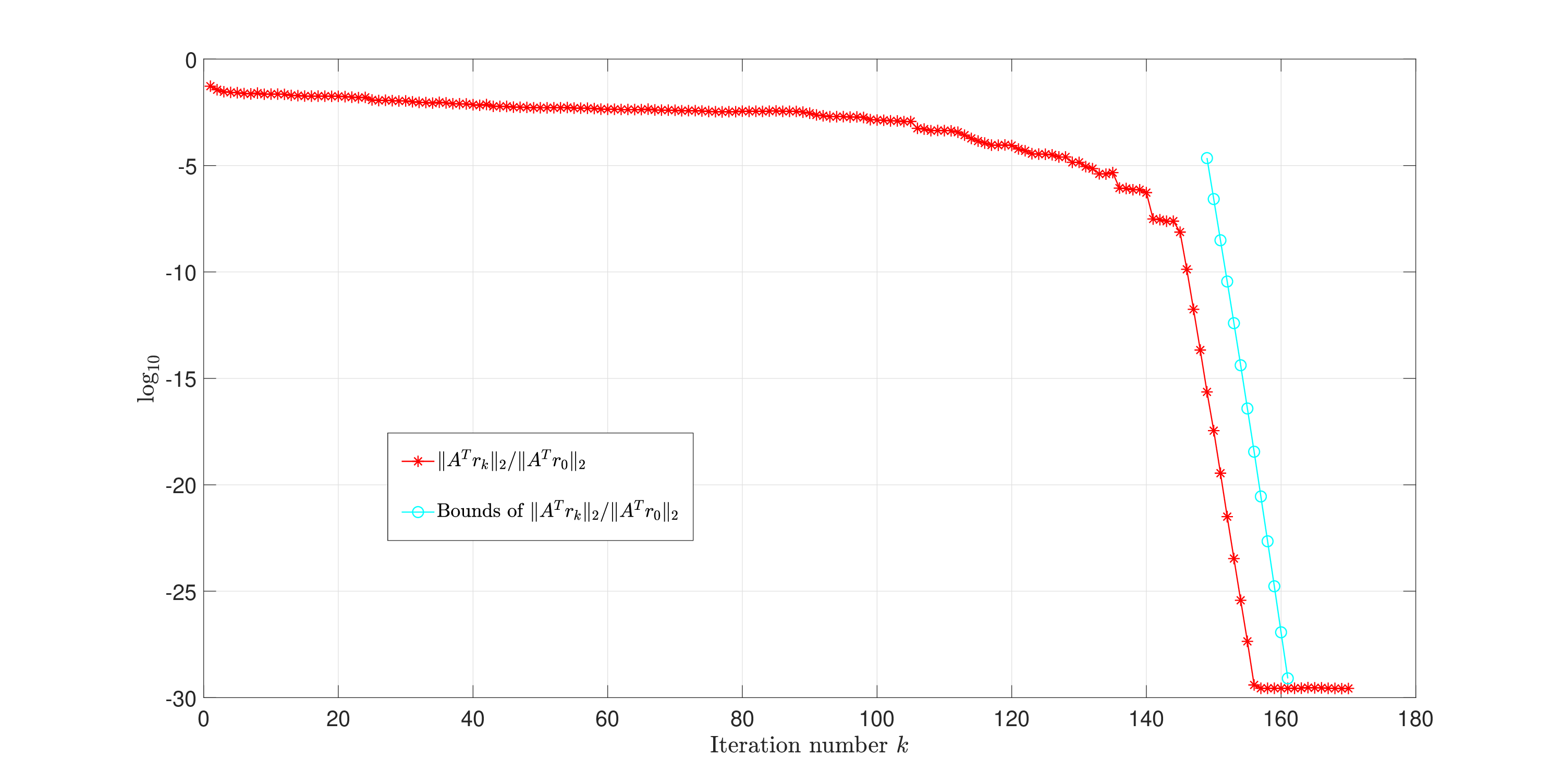}
    \caption{The residual norm and the corresponding residual bound, for Maragal$\_$3T.}
    \label{2025A}
\end{figure}
Next, we tested the Maragal3$\_$T matrix from \cite{florida} with $m=858,n=1,682$ and condition number $2.3484\times 10^{21}$. The result is shown in Figure \ref{2025A}. Here again, the bound is pessimistic in the beginning due to the same reason as above, but it captures the super convergence in the later stage. The bound and the relative residual have the same trend. They reach a very low level of $10^{-32}$ in the quadruple precision computation, although in the initial iterations the bound is more pessimistic compared to Figure \ref{2024A}. 
\section{INFLUENCE OF NON-NORMALITY}
To illustrate how a large condition number of the eigenvectors of $A$ slows down the convergence, especially when the eigenvalues are well
 clustered,
we use the example in \cite{Anne1996Any} by choosing the eigenvalues as $1$, $1.01$, and $1.001$. Then, the characteristic polynomial is given by
\begin{equation*}
    (\lambda-1)(\lambda-1.01)(\lambda-1.001)=\lambda^3-3.011\lambda^2+3.02201\lambda-1.01101,
\end{equation*} and the companion matrix of $A$ is given by
\begin{equation*}
     A^{B}= \left(
  \begin{array}{ccc}
    0 & 0 &  1.0110 \\
    1 & 0 &  -3.02201 \\
    0 &  1&  3.011 \\
  \end{array}
\right),
\end{equation*}
giving non-increasing residual series $\|r_0\|_2=1, \|r_1\|_2=0.99, \|r_2\|_2=0.98$. 
\begin{equation*}
    \boldsymbol{g}=(\sqrt{\|\boldsymbol{r_0}\|_2^2-\|\boldsymbol{r_1}\|_2^2}, \sqrt{\|\boldsymbol{r_1}\|_2^2-\|\boldsymbol{r_2}\|_2^2}, \sqrt{\|\boldsymbol{r_2}\|_2^2-\|\boldsymbol{r_3}\|_2^2})^{\top}=(0.1411,0.1404,0.98)^{\top},
\end{equation*}
where we use $\boldsymbol{g}$ as defined in \cite{Anne1996Any}.\par
Choosing $U$ to be the identity matrix,  $U*\boldsymbol{b}=\boldsymbol{g}$. Thus, $B=(\boldsymbol{b}, \boldsymbol{u_1},\dots,\boldsymbol{u_{n-1}})$ is given as follows.
\begin{equation*}
     B= \left(
  \begin{array}{ccc}
    0.1411 & 1 &  0 \\
    0.1404 & 0 &  1 \\
    0.98 &  0&  0 \\
  \end{array}
\right)
\end{equation*}
Then, we finally have
\begin{equation*}
    A=BA^BB^{-1}=\left(
  \begin{array}{ccc}
    0 & -2.8794 &  1.4329 \\
    1 & 3.1529 &  -0.5957 \\
    0&  0.9908&  -0.1419 \\
  \end{array}\right).
\end{equation*}\par
If we solve $A\boldsymbol{x}=\boldsymbol{b}$ by GMRES, the norm of $\boldsymbol{b}=\boldsymbol{f}(0)$ is $1$, the residual series is $0.99, 0.98.$ Hence, we have
\begin{equation*}
    V_d= \left(
  \begin{array}{ccc}
    0.7776 + 0.0000i  & 0.7776 + 0.0000i & -0.7771 + 0.0000i\\
  -0.4775 - 0.0030i & -0.4775 + 0.0030i &  0.4721 + 0.0000i\\
  -0.4090 + 0.0039i & -0.4090 - 0.0039i &  0.4163 + 0.0000i
  \end{array}
\right),\qquad \kappa(V_d)=8.3057\times 10^{3}.
\end{equation*}
\begin{equation*}
    (c_1, c_2, c_3)^{\top}=[(0.8547 + 1.5162i)\times 10^{3},(0.8547 - 1.5162i)\times 10^{3},(1.7103 - 0.0000i)\times 10^{3})]^{\top},
\end{equation*}
\begin{equation*}
  \|V_d\text{diag}[c_1, c_2, c_3]\|_2=2.9972\times 10^{3}.  
\end{equation*}\par
The Vandermonde part $\|\Lambda_d^1 y_1-[1, 1, \dots, 1]^{\top}\|_2$ in (\ref{eqVandermonde}) for the first iteration is $3.7103\times 10^{-2}$.
Hence, the bound for the first iteration is $\|V_d\text{diag}[c_1, c_2, \dots, c_d]\|\|\Lambda_d^1 \boldsymbol{y_1}-[1, 1, \dots, 1]^{\top}\|_2=1.1120\times 10^2$, which is very large compared to $\|r_1\|_2=0.99$.\par
The Vandermonde part $\|\Lambda_d^2 \boldsymbol{y_2}-[1, 1, \dots, 1]^{\top}\|_2$ in (\ref{eqVandermonde}) for the second iteration is $7.9480\times 10^{-4}$
Hence, the bound for the second iteration is$\|V_d\text{diag}[c_1, c_2, \dots, c_d]\|\|\Lambda_d^2 \boldsymbol{y_2}-[1, 1, \dots, 1]^{\top}\|_2=2.3822$, which is larger than $\|\boldsymbol{r_2}\|_2=0.98$, but the bound is tighter compared to the first iteration. This is because the Vandermonde part converges fast. This implies that the clustering of eigenvalues gives a superlinear bound, although how close the bound and the true convergence curve are, is influenced by $\|V_d\text{diag}[c_1, c_2, \cdots, c_d]\|_2$.\par Applying BA-GMRES to $Ax=b$ with $l$ steps NR-SOR inner-iteration preconditioning with the relaxation parameter $\omega=1.1$, we obtain $\kappa(V_d)=25.69$, $\|V_d\text{diag}[c_1, c_2,  c_3]\|_2=4.28$. As shown in Table \ref{tab:my_labelspecial}, with the increase of $l$, the bound of the residual decreases monotonically and approaches the actual residual. Then, the Vandermonde part of $P^{(5)}A$ for the second iteration is $3.9036\times 10^{-3}.$ Thus, the bound is $1.7079\times10^{-2}$, which is slightly larger than the actual residual $4.1286\times10^{-3}$. This shows that the inner-iteration preconditioning works for this extreme example. \par
\begin{table}[]
    \centering
    \begin{tabular}{c|c|c}
            & 1st Iter.&2nd Iter.\\\hline 
        $l=1$ (Actual) &  $7.0265\times 10^{-1}$& $6.8492\times 10^{-1}$\\
        $l=1$ (Bound)  & 4.2212 & 3.9239\\ \hline
        $l=2$ (Actual)&  $8.7823\times 10^{-1}$ & $8.7720\times 10^{-1}$\\
        $l=2$ (Bound) &  4.0204     &2.3155\\ \hline
        $l=3$ (Actual)&$8.8667\times 10^{-1}$   & $1.4074\times 10^{-1}$\\
        $l=3$ (Bound) & 3.8964 &$4.7470\times 10^{-1}$\\ \hline
          $l=4$ (Actual)&$8.7730\times 10^{-1}$   & $2.1902\times 10^{-2}$\\
        $l=4$ (Bound) & 3.7759 &$8.7574\times 10^{-2}$\\ \hline
         $l=5$ (Actual)&$8.6351\times 10^{-1}$   & $4.1286\times 10^{-3}$\\
        $l=5$ (Bound) & 3.6570 &$1.7079\times 10^{-2}$
    \end{tabular}
    \caption{The comparison of the actual residual $(\|B^{l}\boldsymbol{r_s}\|_2)$ and the bound for BA-GMRES with $l= 1, 2, 3, 4, 5$ steps NR-SOR inner iterations preconditioning. }
    \label{tab:my_labelspecial}
\end{table}
In the inner-iteration preconditioning of least squares problems, one starts with the normal equations so that the original matrix is a normal matrix. It seems that the NR-SOR inner-iterations do not harm the normality (well conditioning of the eigenvectors) a lot according to numerically experiments. The eigenvectors remain unchanged, while the eigenvalues cluster as the inner-iteration steps increases. Thus, the method maintains the condition number of eigenvectors, at the same time improving the distribution of the eigenvalues. That is why the method shows superlinear convergence. 
\section{Conclusions}
This paper was motivated by the superlinear convergence behavior of the inner-iteration preconditioned BA-GMRES. We built a general theory of convergence of GMRES based on the distribution of eigenvalues of the coefficient matrix, especially when they are clustered. Numerical experiments showed that our residual bound captures the sharp decrease of the residual, especially at the later stage of the convergence.  Our way of analyzing gives hints on how the preconditioning works. We will extend the work from the diagonalizable case to the non-diagonalizable case using the Jordan canonical form in the future.



\section*{Acknowledgments}
We would like to thank Dr. Keiichi Morikuni for useful discussions.



\section*{Conflict of interest}
The authors have no conflict of interest to declare.

\bibliography{wileyNJD-Vancouver}

\bibliographystyle{wileyNJD-Vancouver.bst}  

\end{document}